\documentclass[11pt]{article}
\usepackage{amsmath, amssymb, amscd, epsfig, amsthm}
\usepackage{graphicx}
\usepackage{caption, subcaption}

\newtheorem{thm}{Theorem}
\newtheorem{prop}[thm]{Proposition}
\newtheorem{lem}[thm]{Lemma}

\newtheorem{cor}[thm]{Corollary}
\newtheorem{rem}[thm]{Remark}

\renewcommand{\epsilon}{\varepsilon}
\renewcommand{\phi}{\varphi}
\renewcommand{\deg}{\operatorname{deg}}

\newcommand{\BB}{\mathbb}

\newcommand{\separate}{\vskip5pt}

\newcommand{\re}{\operatorname{Re}}
\newcommand{\tr}{\operatorname{Tr}}

\newcommand{\B}{\overline}
\newcommand{\HC}{\BB H_{\BB C}}

\newcommand{\degt}{\widetilde{\operatorname{deg}}}

\usepackage[OT2,T1]{fontenc}
\newcommand\textcyr[1]{{\fontencoding{OT2}\fontfamily{wncyr}\selectfont #1}}
\newcommand{\Zh}{\textit{\textcyr{Zh}}}

\begin{document}

\title{\bf The Conformal Four-Point Integrals, Magic Identities and
Representations of $U(2,2)$}
\author{Matvei Libine}
\maketitle

\begin{abstract}
In \cite{FL1, FL3} we found mathematical interpretations of the one-loop
conformal four-point Feynman integral as well as the vacuum polarization
Feynman integral in the context of representations of a Lie group $U(2,2)$
and quaternionic analysis. Then we raised a natural
question of finding mathematical interpretation of other Feynman diagrams
in the same setting. In this article we describe this interpretation for
all conformal four-point integrals.
Using this interpretation, we give a representation-theoretic proof of
an operator version of the ``magic identities'' for the conformal four-point
integrals described by the box diagrams.

The original ``magic identities'' are due to J.~M.~Drummond, J.~Henn,
V.~A.~Smirnov and E.~Sokatchev; they assert that all $n$-loop box integrals
for four scalar massless particles are equal to each other \cite{DHSS}.
The authors give a proof of the magic identities for the Euclidean metric case
only and claim that the result is also true for the Minkowski metric case.
However, the Minkowski case is much more subtle.
In this article we prove an operator version of the magic identities in the
Minkowski metric case and, in particular, specify the relative positions
of cycles of integration that make these identities correct.

No prior knowledge of physics or Feynman diagrams is assumed from the reader.
We provide a summary of all relevant results from quaternionic analysis to
make the article self-contained.
\end{abstract}



\section{Introduction}

Feynman diagrams are a pictorial way of describing integrals predicting
possible outcomes of interactions of subatomic particles in the context of
quantum field physics.
If at all possible, evaluating these integrals tends to be challenging
and usually produces rather cumbersome expressions.
Moreover, many Feynman diagrams result in integrals that are divergent in
mathematical sense.
Physicists have various techniques called ``renormalizations'' of Feynman
integrals which ``cancel out the infinities'' coming from different parts
of the diagrams.
(For a survey of renormalization techniques see, for example, \cite{Sm}.)
However, these renormalization techniques appear very suspicious to
mathematicians and attract criticism from physicists as well.
For example, if different techniques yield different results,
how do you choose the ``right'' technique?
Or, if they yield the same result, what is the underlying reason for that?
If one can find an intrinsic mathematical meaning of Feynman diagrams and
the corresponding integrals, most of these questions will be resolved.

A number of mathematicians already work on this problem, mostly in
the setting of algebraic geometry.
See, for example, \cite{Mar} for a summary of these algebraic-geometric
developments as well as a comprehensive list of references.
On the other hand, Igor Frenkel has noticed that at least some types of
Feynman diagrams can be interpreted in the context of representation theory
and quaternionic analysis.
In \cite{FL1, FL3, L} we successfully identified the three Feynman diagrams
shown in Figure \ref{basic} with intertwining operators of certain
representations of $U(2,2)$ in the context of quaternionic analysis.
Then we raised a natural question of finding mathematical interpretation
of other Feynman diagrams in the same setting.

\begin{figure}
\begin{center}
\begin{subfigure}{0.25\textwidth}
\centering
\includegraphics[scale=1]{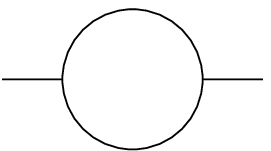}
\end{subfigure}
\begin{subfigure}{0.25\textwidth}
\centering
\includegraphics[scale=1]{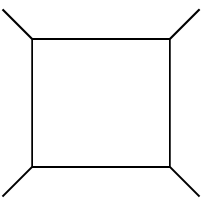}
\end{subfigure}
\begin{subfigure}{0.35\textwidth}
\centering
\includegraphics[scale=1]{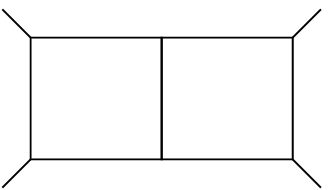}
\end{subfigure}
\end{center}
\caption{Feynman diagrams: the vacuum polarization diagram (left),
the one-loop ladder diagram (center) and the two-loop ladder diagram (right).}
\label{basic}
\end{figure}

This paper deals with conformal four-point integrals described by the
box diagrams. They play an important role in physics, particularly in
Yang-Mills conformal field theory.
For more details see \cite{DHSS} and references therein.
These diagrams have been thoroughly studied by physicists.
For example, the integral described by the one-loop Feynman diagram
is known to express the hyperbolic volume of an ideal tetrahedron,
and is given by the dilogarithm function \cite{DD, W};
there are explicit expressions for the integrals described by the
ladder diagrams in terms of polylogarithms \cite{UD}.
Perhaps the most important property of the conformal four-point integrals
are the ``magic identities'' due to J.~M.~Drummond, J.~Henn, V.~A.~Smirnov
and E.~Sokatchev \cite{DHSS}.
These identities assert that all $n$-loop box integrals
for four scalar massless particles are equal to each other.
We will discuss these ``magic identities'' in Subsection \ref{magic-id-subsect}.

The original paper \cite{DHSS} gives a proof of the magic identities for
the Euclidean metric case only and claims that the result is also true for
the Minkowski metric case.
In the Euclidean case, all variables belong to $\BB H$ and there are
no convergence issues whatsoever.
On the other hand, the Minkowski case (which is the case we consider)
is much more subtle.
In order to deal with convergence issues, we must
consider the so-called ``off-shell Minkowski integrals'' or
perturb the cycles of integration inside $\BB H \otimes \BB C$.
Then the relative position of the cycles becomes very important.
In fact, choosing the ``wrong'' cycles typically results in integral being zero.

In this paper we specify the ``right'' choice of cycles and
find the representation-theoretic meaning of all conformal
four-point integrals. To each such integral, we associate an operator
$L^{(n)}$ on ${\cal H}^+ \otimes {\cal H}^+$, where ${\cal H}^+$ denotes
the space of harmonic functions on the algebra of quaternions $\BB H$.
We prove that the operator $L^{(n)}$ is $\mathfrak{u}(2,2)$-equivariant,
sends ${\cal H}^+ \otimes {\cal H}^+$ into itself and, in particular,
that the result is a function of two variables that is harmonic with respect
to each variable, which is not at all obvious from the construction.
We have a decomposition of  $\mathfrak{u}(2,2)$-representations into
irreducible components:
\begin{equation}  \label{decomp-intro}
(\pi^0_l, {\cal H}^+) \otimes (\pi^0_r, {\cal H}^+) \simeq
\bigoplus_{k=1}^{\infty} (\rho_k,\Zh^+\otimes \BB C^{k \times k}),
\end{equation}
Then, by Schur's Lemma, $L^{(n)}$ acts on each irreducible component
$(\rho_k,\Zh^+\otimes \BB C^{k \times k})$ by multiplication by some scalar
$\mu^{(n)}_k$, and we can find these scalars.
This is the essence of the main result (Theorem \ref{main-thm}).
As an immediate corollary, we obtain the ``magic identities'' for the
operators $L^{(n)}$: Any two box diagrams with the same number of loops
produce the same operator $L^{(n)}$ on ${\cal H}^+ \otimes {\cal H}^+$.
If one can prove that each conformal four-point integral is harmonic
with respect to each variable, then one easily obtains the original
``magic identities'' for the conformal four-point integrals.
The proof of Theorem \ref{main-thm} is essentially by evaluating
the operators $L^{(n)}$ on a suitably chosen set of generators of
${\cal H}^+ \otimes {\cal H}^+$. It is pretty elementary,
and we think that it is an advantage of this approach.

For example, the integrals described by the ladder diagrams have been
evaluated explicitly in \cite{UD}.
The two most simple conformal four-point integrals are the one- and two-loop
ladder integrals $l^{(1)}(Z_1,Z_2;W_1,W_2)$ and $l^{(2)}(Z_1,Z_2;W_1,W_2)$,
which can be expressed in terms of the functions
$$
\Phi^{(1)}(x,y) = \frac1{\lambda} \Bigl( 2 \operatorname{Li}_2(-\rho x)
+ 2 \operatorname{Li}_2(-\rho y)
+ \ln\frac{y}{x} \cdot \ln\frac{1+\rho y}{1+\rho x}
+ \ln(\rho x) \cdot \ln(\rho y) + \frac13 \pi^3 \Bigr),
$$
and
\begin{multline*}
\Phi^{(2)}(x,y) = \frac1{\lambda} \Bigl( 6 \operatorname{Li}_4(-\rho x)
+ 6 \operatorname{Li}_4(-\rho y) + 3 \ln\frac{y}{x} \cdot
\bigl( \operatorname{Li}_3(-\rho x) - \operatorname{Li}_3(-\rho y) \bigr)  \\
+ \frac12 \ln^2\frac{y}{x}
\cdot \bigl( \operatorname{Li}_2(-\rho x) - \operatorname{Li}_2(-\rho y) \bigr)
+ \frac14 \ln^2(\rho x) \cdot \ln^2(\rho y)  \\
+ \frac12 \pi^2 \ln(\rho x) \cdot \ln(\rho y)
+ \frac1{12} \pi^2 \ln\frac{y}{x} + \frac7{60} \pi^4 \Bigr)
\end{multline*}
respectively, where
$$
\lambda(x,y) = \sqrt{ (1-x-y)^2 - 4xy}, \qquad
\rho(x,y) = \frac2{1-x-y+\lambda},
$$
and $\operatorname{Li}_N$ denotes the polylogarithm function:
$$
\operatorname{Li}_N(z) = \frac{(-1)^N}{(N-1)!}
\int_0^1 \frac{\ln^{N-1} \xi}{\xi - z^{-1}} \,d\xi.
$$
The expressions for the other ladder integrals are similar.

By contrast, we have very simple expressions for the operators
$L^{(1)}$ and $L^{(2)}$ on ${\cal H}^+ \otimes {\cal H}^+$.
The operator $L^{(1)}$ is just the projection of ${\cal H}^+ \otimes {\cal H}^+$
onto its first irreducible component $(\rho_1,\Zh^+)$ in the decomposition
(\ref{decomp-intro}).
And $L^{(2)}$ acts on each irreducible component of
${\cal H}^+ \otimes {\cal H}^+$ by multiplication by a scalar, so that if
$x \in {\cal H}^+ \otimes {\cal H}^+$ belongs to an irreducible component
isomorphic to $(\rho_k,\Zh^+ \otimes \BB C^{k \times k})$
in the decomposition (\ref{decomp-intro}), then
$$
L^{(2)}(x) = \mu^{(2)}_k x, \qquad \text{where} \qquad
\mu^{(2)}_k =
\begin{cases}
1 & \text{if $k=1$;} \\
\frac{(-1)^{k+1}}{k(k-1)} & \text{if $k \ge 2$.}
\end{cases}
$$

Thus we have a representation-theoretic interpretation of an infinite family
of Feynman diagrams, and it is reasonable to expect that an even larger class
of Feynman diagrams can be interpreted in the same context.
Finally, we comment that it is not really necessary to use quaternionic
setting to interpret the box diagrams and the corresponding integrals
-- one could do the same in the setting of analytic functions of
four variables instead.
However, the vacuum polarization diagram does require quaternionic analysis.
Also, this article uses results that have already been stated and proved
in quaternionic setting. For these reasons we continue to use quaternions.

The paper is organized as follows. In Section \ref{preliminaries}
we establish our notations and state relevant results from quaternionic
analysis. In Section \ref{results-summary} we state more recent results
from \cite{FL3} and \cite{L} that are used in the proofs.
In Section \ref{fd-section} we review the box diagrams and the corresponding
conformal four-point integrals, state the magic identities and the main result
(Theorem \ref{main-thm}).
In Section \ref{proof-section} we prove Theorem \ref{main-thm}, first,
in the case of ladder diagrams, and then in general.

\section{Preliminaries}  \label{preliminaries}

In this section we establish notations and state relevant results from
quaternionic analysis. We mostly follow our previous papers \cite{FL1},
\cite{FL2} and \cite{L}.
A contemporary review of quaternionic analysis can be found in \cite{Su}.
Quaternionic analysis also has many applications in physics
(see, for instance, \cite{GT}).

\subsection{Complexified Quaternions $\HC$ and the Conformal Group $GL(2,\HC)$}

We recall some notations from \cite{FL1}.
Let $\HC$ denote the space of complexified quaternions:
$\HC = \BB H \otimes \BB C$, it can be identified with the algebra of
$2 \times 2$ complex matrices:
$$
\HC = \BB H \otimes \BB C \simeq \biggl\{
Z = \begin{pmatrix} z_{11} & z_{12} \\ z_{21} & z_{22} \end{pmatrix}
; \: z_{ij} \in \BB C \biggr\}
= \biggl\{ Z= \begin{pmatrix} z^0-iz^3 & -iz^1-z^2 \\ -iz^1+z^2 & z^0+iz^3
\end{pmatrix} ; \: z^k \in \BB C \biggr\}.
$$
For $Z \in \HC$, we write
$$
N(Z) = \det \begin{pmatrix} z_{11} & z_{12} \\ z_{21} & z_{22} \end{pmatrix}
= z_{11}z_{22}-z_{12}z_{21} = (z^0)^2 + (z^1)^2 + (z^2)^2 + (z^3)^2
$$
and think of it as the norm of $Z$.
We realize $U(2)$ as
$$
U(2) = \{ Z \in \HC ;\: Z^*=Z^{-1} \},
$$
where $Z^*$ denotes the complex conjugate transpose of a complex matrix $Z$.
For $R>0$, we set
$$
U(2)_R = \{ RZ ;\: Z \in U(2) \} \quad \subset \HC
$$
and orient it as in \cite{FL1}, so that
$$
\int_{U(2)_R} \frac{dV}{N(Z)^2} = -2\pi^3 i,
$$
where $dV$ is a holomorphic 4-form
$$
dV = dz^0 \wedge dz^1 \wedge dz^2 \wedge dz^3
= \frac14 dz_{11} \wedge dz_{12} \wedge dz_{21} \wedge dz_{22}.
$$
Recall that a group $GL(2,\HC) \simeq GL(4,\BB C)$ acts on $\HC$ by fractional
linear (or conformal) transformations:
\begin{equation}  \label{conformal-action}
h: Z \mapsto (aZ+b)(cZ+d)^{-1} = (a'-Zc')^{-1}(-b'+Zd'),
\qquad Z \in \HC,
\end{equation}
where
$h = \bigl(\begin{smallmatrix} a & b \\ c & d \end{smallmatrix}\bigr)
\in GL(2,\HC)$ and 
$h^{-1} = \bigl(\begin{smallmatrix} a' & b' \\ c' & d' \end{smallmatrix}\bigr)$.

\subsection{Harmonic Functions on $\HC$}

As in Section 2 of \cite{FL2}, we consider the space $\widetilde{\cal H}$
consisting of $\BB C$-valued functions on $\HC$ (possibly with singularities)
that are holomorphic with respect to the complex variables
$z_{11},z_{12},z_{21},z_{22}$ and harmonic, i.e. annihilated by
$$
\square 
= 4\biggl( \frac{\partial^2}{\partial z_{11}\partial z_{22}}
- \frac{\partial^2}{\partial z_{12}\partial z_{21}} \biggr)
= \frac{\partial^2}{(\partial z^0)^2} + \frac{\partial^2}{(\partial z^1)^2}
+ \frac{\partial^2}{(\partial z^2)^2} + \frac{\partial^2}{(\partial z^3)^2}.
$$
Then the conformal group $GL(2,\HC)$ acts on $\widetilde{\cal H}$ by two
slightly different actions:
\begin{align*}
\pi^0_l(h): \: \phi(Z) \quad &\mapsto \quad \bigl( \pi^0_l(h)\phi \bigr)(Z) =
\frac 1{N(cZ+d)} \cdot \phi \bigl( (aZ+b)(cZ+d)^{-1} \bigr),  \\
\pi^0_r(h): \: \phi(Z) \quad &\mapsto \quad \bigl( \pi^0_r(h)\phi \bigr)(Z) =
\frac 1{N(a'-Zc')} \cdot \phi \bigl( (a'-Zc')^{-1}(-b'+Zd') \bigr),
\end{align*}
where
$h = \bigl(\begin{smallmatrix} a' & b' \\ c' & d' \end{smallmatrix}\bigr)
\in GL(2,\HC)$ and
$h^{-1} = \bigl(\begin{smallmatrix} a & b \\ c & d \end{smallmatrix}\bigr)$.
These two actions coincide on $SL(2,\HC) \simeq SL(4,\BB C)$
which is defined as the connected Lie subgroup of $GL(2,\HC)$ with Lie algebra
$$
\mathfrak{sl}(2,\HC) = \{ x \in \mathfrak{gl}(2,\HC) ;\: \re (\tr x) =0 \}
\simeq \mathfrak{sl}(4,\BB C).
$$

We introduce two spaces of harmonic polynomials:
$$
{\cal H}^+ = \widetilde{\cal H} \cap \BB C[z_{11},z_{12},z_{21},z_{22}],
$$
$$
{\cal H} = \widetilde{\cal H} \cap \BB C[z_{11},z_{12},z_{21},z_{22}, N(Z)^{-1}]
$$
and the space of harmonic polynomials regular at infinity:
$$
{\cal H}^- = \bigl\{ \phi \in \widetilde{\cal H};\:
N(Z)^{-1} \cdot \phi(Z^{-1}) \in {\cal H}^+ \bigr\}.
$$
Then
$$
{\cal H} = {\cal H}^- \oplus {\cal H}^+.
$$
Differentiating the actions $\pi^0_l$ and $\pi^0_r$, we obtain actions of
$\mathfrak{gl}(2,\HC) \simeq \mathfrak{gl}(4,\BB C)$ which preserve
the spaces ${\cal H}$, ${\cal H}^-$ and ${\cal H}^+$.
By abuse of notation, we denote these Lie algebra actions by
$\pi^0_l$ and $\pi^0_r$ respectively.
They are described in Subsection 3.2 of \cite{FL2}.

By Theorem 28 in \cite{FL1}, for each $R>0$, we have a
bilinear pairing between $(\pi^0_l, {\cal H})$ and $(\pi^0_r, {\cal H})$:
\begin{equation}  \label{H-pairing}
(\phi_1,\phi_2)_R = \frac 1{2\pi^2}
\int_{Z \in S^3_R} (\degt \phi_1)(Z) \cdot \phi_2(Z) \,\frac{dS}R,
\qquad \phi_1, \phi_2 \in {\cal H},
\end{equation}
where $S^3_R \subset \BB H$ is the three-dimensional sphere of radius $R$
centered at the origin
$$
S^3_R = \{ X \in \BB H ;\: N(X)=R^2 \},
$$
$dS$ denotes the usual Euclidean volume element on $S^3_R$, and
$\degt$ denotes the degree operator plus identity:
$$
\degt f = f + \deg f = f + z_{11}\frac{\partial f}{\partial z_{11}} +
z_{12}\frac{\partial f}{\partial z_{12}} + z_{21}\frac{\partial f}{\partial z_{21}}
+ z_{22}\frac{\partial f}{\partial z_{22}}.
$$
When this pairing is restricted to ${\cal H}^+ \times {\cal H}^-$,
it is $\mathfrak{gl}(2,\HC)$-invariant, independent of the choice of $R>0$,
non-degenerate and antisymmetric
$$
(\phi_1,\phi_2)_R = - (\phi_2,\phi_1)_R,
\qquad \phi_1 \in {\cal H}^+, \: \phi_2 \in {\cal H}^-.
$$

When restricted to $\mathfrak{u}(2,2)$, the representations
$(\pi^0_l, {\cal H}^+)$ and $(\pi^0_r, {\cal H}^+)$ become irreducible unitary
with respect to the inner product
\begin{equation}  \label{inner-prod}
\langle \phi_1,\phi_2 \rangle_{inn.\: prod.} =
\int_{Z \in S^3_1} (\widetilde{\deg} \phi_1)(Z) \cdot \B{\phi_2}(Z) \,dS,
\qquad \phi_1,\phi_2 \in {\cal H}^+,
\end{equation}
(Theorem 28 in \cite{FL1}).

We conclude this subsection with an analogue of the Poisson formula
(Theorem 34 in \cite{FL1}). It involves a certain open region $\BB D^+_R$
in $\HC$ which will be defined in (\ref{D_R}).

\begin{thm} \label{Poisson}
Let $R>0$ and let $\phi \in \widetilde{\cal H}$ be a harmonic function
with no singularities on the closure of $\BB D^+_R$, then
$$
\phi(W) = \biggl( \phi,\frac1{N(Z-W)} \biggr)_R =
\frac 1{2\pi^2} \int_{Z \in S^3_R} \frac{(\degt \phi)(Z)}{N(Z-W)}
\,\frac{dS}R, \qquad \forall W \in \BB D^+_R.
$$
\end{thm}

\subsection{Representation $(\rho_1,\Zh)$ of $\mathfrak{gl}(2,\HC)$}

Let $\widetilde{\Zh}$ denote the space of $\BB C$-valued functions on $\HC$
(possibly with singularities) which are holomorphic with respect to the
complex variables $z_{11}$, $z_{12}$, $z_{21}$, $z_{22}$.
We recall the action of $GL(2,\HC)$ on $\widetilde{\Zh}$ given by
equation (49) in \cite{FL1}:
\begin{equation*}
\rho_1(h): \: f(Z) \quad \mapsto \quad \bigl( \rho_1(h)f \bigr)(Z) =
\frac {f \bigl( (aZ+b)(cZ+d)^{-1} \bigr)}{N(cZ+d) \cdot N(a'-Zc')},
\end{equation*}
where
$h = \bigl(\begin{smallmatrix} a' & b' \\ c' & d' \end{smallmatrix}\bigr)
\in GL(2,\HC)$ and 
$h^{-1} = \bigl(\begin{smallmatrix} a & b \\ c & d \end{smallmatrix}\bigr)$.
Differentiating the $\rho_1$-action, we obtain an action
(still denoted by $\rho_1$) of $\mathfrak{gl}(2,\HC)$ which preserves spaces
\begin{align}
\Zh^+ &= \{\text{polynomial functions on $\HC$}\}
= \BB C[z_{11},z_{12},z_{21},z_{22}] \qquad \text{and}  \label{zh+} \\
\Zh &= \bigl\{\text{polynomial functions on
$\{ Z \in \HC ;\: N(Z) \ne 0 \}$}\bigr\}
= \BB C[z_{11},z_{12},z_{21},z_{22}, N(Z)^{-1}].  \label{zh}
\end{align}

Recall Proposition 69 from \cite{FL1}:

\begin{prop}
The representation $(\rho_1,\Zh)$ of $\mathfrak{gl}(2,\HC)$
has a non-degenerate symmetric bilinear pairing
\begin{equation}  \label{pairing}
\langle f_1,f_2 \rangle =
\frac i{2\pi^3} \int_{Z \in U(2)_R} f_1(Z) \cdot f_2(Z) \,dV,
\qquad f_1, f_2 \in \Zh.
\end{equation}
This bilinear pairing is $\mathfrak{gl}(2,\HC)$-invariant and
independent of the choice of $R>0$.
\end{prop}

\subsection{The Group $\HC^{\times}$ and Its Matrix Coefficients}  \label{matrix-coeff-subsection}

We denote by $\HC^{\times}$ the group of invertible complexified quaternions:
$$
\HC^{\times} = \{ Z \in \HC ;\: N(Z) \ne 0 \} \simeq GL(2,\BB C).
$$
We denote by $(\tau_{\frac12},\BB S)$ the tautological 2-dimensional
representation of $\HC^{\times}$.
Then, for $l=0,\frac12,1,\frac32, \dots$, we denote by $(\tau_l,V_l)$ the
$2l$-th symmetric power product of $(\tau_{\frac12},\BB S)$.
(In particular, $(\tau_0,V_0)$ is the trivial one-dimensional representation.)
Thus, each $(\tau_l,V_l)$ is an irreducible representation of $\HC^{\times}$
of dimension $2l+1$.
A concrete realization of $(\tau_l,V_l)$ as well as an isomorphism
$V_l \simeq \BB C^{2l+1}$ suitable for our purposes are described in
Subsection 2.5 of \cite{FL1}.

Recall the matrix coefficient functions of $\tau_l(Z)$ described by
equation (27) of \cite{FL1} (cf. \cite{V}):
\begin{equation*}  
t^l_{n\,\underline{m}}(Z) = \frac 1{2\pi i}
\oint (sz_{11}+z_{21})^{l-m} (sz_{12}+z_{22})^{l+m} s^{-l+n} \,\frac{ds}s,
\qquad
\begin{matrix} l = 0, \frac12, 1, \frac32, \dots, \\ m,n \in \BB Z +l, \\
 -l \le m,n \le l, \end{matrix}
\end{equation*}
$Z=\bigl(\begin{smallmatrix} z_{11} & z_{12} \\
z_{21} & z_{22} \end{smallmatrix}\bigr) \in \HC$,
the integral is taken over a loop in $\BB C$ going once around the origin
in the counterclockwise direction.
We regard these functions as polynomials on $\HC$. For example,
\begin{equation}  \label{t-special}
t^l_{-l\,\underline{-l}}(Z) = (z_{11})^{2l}, \qquad
t^l_{-l\,\underline{l}}(Z) = (z_{12})^{2l}, \qquad
t^l_{l\,\underline{-l}}(Z) = (z_{21})^{2l}, \qquad
t^l_{l\,\underline{l}}(Z) = (z_{22})^{2l}.
\end{equation}


We have the following orthogonality relations with respect to the pairing
(\ref{H-pairing}):
\begin{equation}  \label{H-orthogonality}
\bigl( t^{l'}_{n'\,\underline{m'}}(Z), t^l_{m\underline{n}}(Z^{-1}) \cdot N(Z)^{-1} \bigr)_R
= -\bigl(t^l_{m\underline{n}}(Z^{-1}) \cdot N(Z)^{-1},t^{l'}_{n'\,\underline{m'}}(Z)\bigr)_R
= \delta_{ll'} \delta_{mm'} \delta_{nn'},
\end{equation}
the following orthogonality relations with respect to the inner product
(\ref{inner-prod}):
\begin{equation}  \label{H-unitary-orthogonality}
\bigl\langle t^l_{n\underline{m}}(Z), t^{l'}_{n'\,\underline{m'}}(Z)
\bigr\rangle_{inn.\: prod.}
= \frac{(l-m)!(l+m)!}{(l-n)!(l+n)!} \delta_{ll'} \delta_{mm'} \delta_{nn'},
\end{equation}
and similar orthogonality relations with respect to the pairing (\ref{pairing}):
\begin{equation}  \label{orthogonality}
\bigl\langle t^{l'}_{n'\,\underline{m'}}(Z) \cdot N(Z)^{k'},
t^l_{m\underline{n}}(Z^{-1}) \cdot N(Z)^{-k-2} \bigr\rangle
= \frac1{2l+1} \delta_{kk'}\delta_{ll'} \delta_{mm'} \delta_{nn'},
\end{equation}
where the indices $k,l,m,n$ are
$l = 0, \frac12, 1, \frac32, \dots$, $m,n \in \BB Z +l$, $-l \le m,n \le l$,
$k \in \BB Z$ and similarly for $k',l',m',n'$ (see, for example, \cite{V}).

One advantage of working with these functions is that they form $K$-type bases
of various spaces:

\begin{prop} [Proposition 19 in \cite{FL1}, Proposition 5 in \cite{FL3} and
Corollary 6 in \cite{FL3}]
\begin{enumerate}
\item
The functions 
$$
t^l_{n\,\underline{m}}(Z), \qquad
l=0, \frac12, 1, \frac32, \dots, \quad m,n=-l,-l+1,\dots,l,
$$
form a vector space basis of
${\cal H}^+ = \{ \phi \in \Zh^+;\: \square\phi=0 \}$;
\item
The functions 
$$
t^l_{n\,\underline{m}}(Z) \cdot N(Z)^{-(2l+1)}, \qquad
l=0, \frac12, 1, \frac32, \dots, \quad m,n=-l,-l+1,\dots,l,
$$
form a vector space basis of ${\cal H}^-$;
\item
The functions 
$$
t^l_{n\,\underline{m}}(Z) \cdot N(Z)^k, \qquad
l=0, \frac12, 1, \frac32, \dots, \quad m,n=-l,-l+1,\dots,l, \quad k=0,1,2,\dots,
$$
form a vector space basis of $\Zh^+ = \BB C[z_{11},z_{12},z_{21},z_{22}]$;
\item
The functions 
\begin{equation}  \label{Zh-basis}
t^l_{n\,\underline{m}}(Z) \cdot N(Z)^k, \qquad
l=0, \frac12, 1, \frac32, \dots, \quad m,n=-l,-l+1,\dots,l, \quad k \in\BB Z,
\end{equation}
form a vector space basis of $\Zh = \BB C[z_{11},z_{12},z_{21},z_{22}, N(Z)^{-1}]$.
\end{enumerate}
\end{prop}

Another advantage is having matrix coefficient expansions such as
those described in Propositions 25, 26 and 27 in \cite{FL1}.
For convenience we restate Proposition 25 from \cite{FL1}:

\begin{prop}
We have the following matrix coefficient expansion
\begin{equation}  \label{1/N-expansion}
\frac 1{N(Z-W)}= N(W)^{-1} \cdot \sum_{l,m,n}
t^l_{m\,\underline{n}}(Z) \cdot t^l_{n\,\underline{m}}(W^{-1}),
\qquad \begin{matrix} l=0,\frac 12, 1, \frac 32,\dots, \\
m,n = -l, -l+1, \dots, l, \end{matrix}
\end{equation}
which converges pointwise absolutely in the region
$\{ (Z,W) \in \HC \times \HC^{\times}; \: ZW^{-1} \in \BB D^+ \}$,
where $\BB D^+$ is an open region in $\HC$ to be defined in (\ref{D}).
\end{prop}

\subsection{Subgroups $U(2,2)_R \subset GL(2,\HC)$ and Domains
$\BB D^+_R$, $\BB D^-_R$}

We often regard the group $U(2,2)$ as a subgroup of $GL(2,\HC)$,
as described in Subsection 3.5 of \cite{FL1}. That is
$$
U(2,2) = \Biggl\{ \begin{pmatrix} a & b \\ c & d \end{pmatrix}
\in GL(2,\HC) ;\: a,b,c,d \in \HC,\:
\begin{matrix} a^*a = 1+c^*c \\ d^*d = 1+b^*b \\ a^*b=c^*d \end{matrix}
\Biggr\}.
$$
The maximal compact subgroup of $U(2,2)$ is
\begin{equation*}  
U(2) \times U(2) = \biggl\{
\begin{pmatrix} a & 0 \\ 0 & d \end{pmatrix} \in GL(2, \HC);\:
a,d \in \HC, \: a^*a=d^*d=1 \biggr\}.
\end{equation*}

The group $U(2,2)$ acts on $\HC$ by fractional linear transformations
(\ref{conformal-action}) preserving $U(2) \subset \HC$ and open domains
\begin{equation}  \label{D}
\BB D^+ = \{ Z \in \HC;\: ZZ^*<1 \}, \qquad
\BB D^- = \{ Z \in \HC;\: ZZ^*>1 \},
\end{equation}
where the inequalities $ZZ^*<1$ and $ZZ^*>1$ mean that the matrix $ZZ^*-1$
is negative and positive definite respectively.
The sets $\BB D^+$ and $\BB D^-$ both have $U(2)$ as the Shilov boundary.

Similarly, for each $R>0$, we can define a conjugate of $U(2,2)$
$$
U(2,2)_R = \begin{pmatrix} R & 0 \\ 0 & 1 \end{pmatrix} U(2,2)
\begin{pmatrix} R^{-1} & 0 \\ 0 & 1 \end{pmatrix} \quad \subset GL(2,\HC).
$$
Each group $U(2,2)_R$ is a real form of $GL(2,\HC)$, preserves $U(2)_R$
and open domains
\begin{equation}  \label{D_R}
\BB D^+_R = \{ Z \in \HC ;\: ZZ^*<R^2 \}, \qquad
\BB D^-_R = \{ Z \in \HC ;\: ZZ^*>R^2 \}.
\end{equation}
These sets $\BB D^+_R$ and $\BB D^-_R$ both have $U(2)_R$ as the Shilov boundary.

\section{Summary of Results from \cite{FL3} and \cite{L}}  \label{results-summary}

\subsection{Irreducible Components of $(\rho_1,\Zh)$
and Equivariant Maps  \\
$(\rho_1,\Zh) \to (\pi_l^0, {\cal H}) \otimes (\pi_r^0,{\cal H})$}

First, we state the decomposition theorem:

\begin{thm}[Theorem 7 in \cite{FL3}]
The representation $(\rho_1,\Zh)$ of $\mathfrak{gl}(2,\HC)$ has the
following decomposition into irreducible components:
$$
(\rho_1,\Zh) = (\rho_1,\Zh^-) \oplus (\rho_1,\Zh^0) \oplus (\rho_1,\Zh^+),
$$
where
\begin{align*}
\Zh^+ &= \BB C\text{-span of }
\bigl\{ t^l_{n\,\underline{m}}(Z) \cdot N(Z)^k;\: k \ge 0 \bigr\}, \\
\Zh^- &= \BB C\text{-span of }
\bigl\{ t^l_{n\,\underline{m}}(Z) \cdot N(Z)^k;\: k \le -(2l+2) \bigr\}, \\
\Zh^0 &= \BB C\text{-span of }
\bigl\{ t^l_{n\,\underline{m}}(Z) \cdot N(Z)^k;\: -(2l+1) \le k \le -1 \bigr\}
\end{align*}
(see Figure \ref{decomposition-fig}).
\end{thm}

\begin{figure}
\begin{center}
\setlength{\unitlength}{1mm}
\begin{picture}(110,70)
\multiput(10,10)(10,0){10}{\circle*{1}}
\multiput(10,20)(10,0){10}{\circle*{1}}
\multiput(10,30)(10,0){10}{\circle*{1}}
\multiput(10,40)(10,0){10}{\circle*{1}}
\multiput(10,50)(10,0){10}{\circle*{1}}
\multiput(10,60)(10,0){10}{\circle*{1}}
\thicklines
\put(60,0){\vector(0,1){70}}
\put(0,10){\vector(1,0){110}}
\thinlines
\put(58,10){\line(0,1){55}}
\put(60,8){\line(1,0){45}}
\put(52,10){\line(0,1){55}}
\put(5,8){\line(1,0){35}}
\put(41.4,11.4){\line(-1,1){36.4}}
\put(48.6,8.6){\line(-1,1){43.6}}
\qbezier(58,10)(58,8)(60,8)
\qbezier(40,8)(43.8,8)(41.4,11.4)
\qbezier(52,10)(52,5.2)(48.6,8.6)
\put(62,67){$2l$}
\put(107,12){$k$}
\put(3,25){$\Zh^-$}
\put(23,62){$\Zh^0$}
\put(102,44){$\Zh^+$}
\end{picture}
\end{center}
\caption{Decomposition of $(\rho_1,\Zh)$ into irreducible components.}
\label{decomposition-fig}
\end{figure}
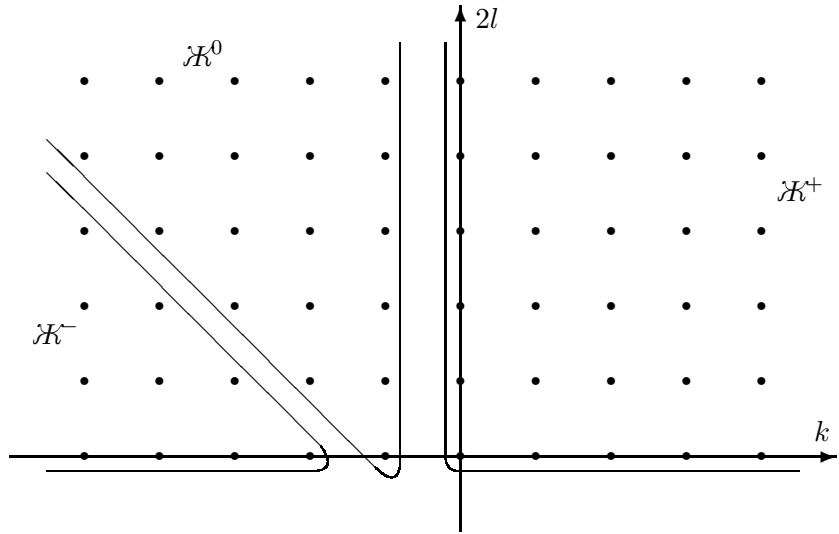

A tensor product $(\pi_l^0, {\cal H}^+) \otimes (\pi_r^0,{\cal H}^+)$ of
representations of $\mathfrak{gl}(2,\HC)$ decomposes into a direct sum
of irreducible subrepresentations, one of which is $(\rho_1,\Zh^+)$.
This decomposition is stated precisely in equation (\ref{tensor-decomp}).
The irreducible component $(\rho_1,\Zh^+)$ has multiplicity one
and is generated by $1 \otimes 1 \in {\cal H}^+ \otimes {\cal H}^+$.
Thus we have a $\mathfrak{gl}(2,\HC)$-equivariant map
\begin{equation*}
I: (\rho_1,\Zh^+) \hookrightarrow
(\pi_l^0, {\cal H}^+) \otimes (\pi_r^0,{\cal H}^+),
\end{equation*}
which is unique up to multiplication by a scalar.
This scalar can be pinned down by a requirement $I(1) = 1 \otimes 1$.

We consider a map
\begin{equation}  
\Zh \ni f(Z) \quad \mapsto \quad (I_R f)(W_1,W_2) =
\frac i{2\pi^3} \int_{Z \in U(2)_R} \frac{f(Z) \,dV}{N(Z-W_1) \cdot N(Z-W_2)}
\quad \in \B{{\cal H} \otimes {\cal H}},
\end{equation}
where $\B{{\cal H} \otimes {\cal H}}$ denotes the Hilbert space obtained by
completing ${\cal H} \otimes {\cal H}$ with respect to the unitary structure
coming from the tensor product of unitary representations
$(\pi^0_l, {\cal H})$ and $(\pi^0_r, {\cal H})$.
If $W_1, W_2 \in \BB D^+_R$ or $W_1, W_2 \in \BB D^-_R$, the integrand has
no singularities and the result is a holomorphic function in two variables
$W_1, W_2$ which is harmonic in each variable separately.

\begin{thm}[[Theorem 12 and Corollary 14 in \cite{FL3}]  
When $W_1, W_2 \in \BB D^+_R$, the map $I_R$ annihilates $\Zh^- \oplus \Zh^0$,
and its restriction to $\Zh^+$ coincides with the map $I$.

When $W_1, W_2 \in \BB D^-_R$, the map $I_R$ annihilates $\Zh^0 \oplus \Zh^+$,
and its restriction to $\Zh^-$ produces an equivariant embedding
$(\rho_1,\Zh^-) \hookrightarrow
(\pi_l^0, {\cal H}^-) \otimes (\pi_r^0,{\cal H}^-)$.

\end{thm}

Next we have a lemma that will be used for evaluating integral
operators $L^{(n)}$ on the generators of
$(\pi_l^0,{\cal H}^+) \otimes (\pi_r^0, {\cal H}^+)$.

\begin{lem}[Lemma 18 in \cite{L}]  \label{z^p}
Let $k=1,2,3,\dots$, and let $z_{ij}$ be $z_{11}$, $z_{12}$, $z_{21}$ or $z_{22}$. Then
$$
\bigl( I (z_{ij})^k \bigr)(W,W') =
\frac1{k+1} \sum_{p=0}^k (w_{ij})^p \cdot (w'_{ij})^{k-p}.
$$
\end{lem}

Finally, we have the following consequence of the proof of this lemma.

\begin{cor} [Corollary 19 in \cite{L}]  \label{z^p-cor}
Let $k \ge 0$. We have the following orthogonality relations:
$$
\bigl\langle N(Z)^{-2-k} \cdot t^l_{m \, \underline{n}}(Z^{-1}) \cdot
t^{l'}_{m' \, \underline{n'}}(Z^{-1}), t^{p/2}_{-p/2\,\underline{-p/2}}(Z) \bigr\rangle =
\begin{cases}
\frac1{p+1} & \text{if $k=0$, $l+l'=p/2$,} \\
& \text{$m=n=-l$ and $m'=n'=-l'$}; \\
0 & \text{otherwise}.
\end{cases}
$$
\end{cor}

\subsection{Representations $(\varpi_2^l,\Zh)$, $(\varpi_2^r,\Zh)$ and
Their Subrepresentations}

Recall that $\widetilde{\Zh}$ denotes the space of $\BB C$-valued functions
on $\HC$ (possibly with singularities) which are holomorphic with respect to
the complex variables $z_{11}$, $z_{12}$, $z_{21}$, $z_{22}$.
We define two very similar actions of $GL(2,\HC)$ on $\widetilde{\Zh}$:
\begin{align*}
\varpi_2^l(h): \: f(Z) \quad &\mapsto \quad \bigl( \varpi_2^l(h)f \bigr)(Z) =
\frac {f \bigl( (aZ+b)(cZ+d)^{-1} \bigr)}{N(cZ+d)^2 \cdot N(a'-Zc')},  \\
\varpi_2^r(h): \: f(Z) \quad &\mapsto \quad \bigl( \varpi_2^r(h)f \bigr)(Z) =
\frac {f \bigl( (aZ+b)(cZ+d)^{-1} \bigr)}{N(cZ+d) \cdot N(a'-Zc')^2},
\end{align*}
where
$h = \bigl(\begin{smallmatrix} a' & b' \\ c' & d' \end{smallmatrix}\bigr)
\in GL(2,\HC)$ and 
$h^{-1} = \bigl(\begin{smallmatrix} a & b \\ c & d \end{smallmatrix}\bigr)$.
(These actions coincide on $SL(2,\HC)$.)
Note that $\varpi_2^l$ is the action $\varpi_2$ in the notations of \cite{L}.
Differentiating, we obtain actions of $\mathfrak{gl}(2,\HC)$ which preserve
the spaces $\Zh$ and $\Zh^+$ defined by (\ref{zh+})-(\ref{zh}).

\begin{thm}[Theorem 8 in \cite{L}]
The spaces
\begin{align*}
\Zh^+ &= \BB C \text{-span of }
\bigl\{ t^l_{n\,\underline{m}}(Z) \cdot N(Z)^k;\: k \ge 0 \bigr\}, \\
\Zh_2^- &= \BB C \text{-span of }
\bigl\{ t^l_{n\,\underline{m}}(Z) \cdot N(Z)^k;\: k \le -(2l+3) \bigr\}, \\
I_2^- &= \BB C \text{-span of }
\bigl\{ t^l_{n\,\underline{m}}(Z) \cdot N(Z)^k;\: k \le -2 \bigr\}, \\
I_2^+ &= \BB C \text{-span of }
\bigl\{ t^l_{n\,\underline{m}}(Z) \cdot N(Z)^k;\: k \ge -(2l+1) \bigr\}, \\
J_2 &= \BB C \text{-span of }
\bigl\{ t^l_{n\,\underline{m}}(Z) \cdot N(Z)^k;\: -(2l+1) \le k \le -2 \bigr\}
\end{align*}
and their sums are the only proper subspaces of $\Zh$ that are invariant
under either $\varpi_2^l$ or $\varpi_2^r$ actions of $\mathfrak{gl}(2,\HC)$
(see Figure \ref{decomposition-fig2}).

The irreducible components of $(\varpi_2^l, \Zh)$ and $(\varpi_2^r, \Zh)$
are the subrepresentations
$$
(\varpi_2^*, \Zh^+), \qquad (\varpi_2^*, \Zh_2^-), \qquad (\varpi_2^*, J_2)
$$
and the quotients
$$
\bigl( \varpi_2^*, \Zh/(I_2^- \oplus \Zh^+) \bigr)
= \bigl( \varpi_2^*, I_2^+/(\Zh^+ \oplus J_2) \bigr), \quad
\bigl( \varpi_2^*, \Zh/(\Zh_2^- \oplus I_2^+) \bigr)
= \bigl( \varpi_2^*, I_2^-/(\Zh_2^- \oplus J_2) \bigr),
$$
where $*$ stands for $l$ or $r$.
\end{thm}

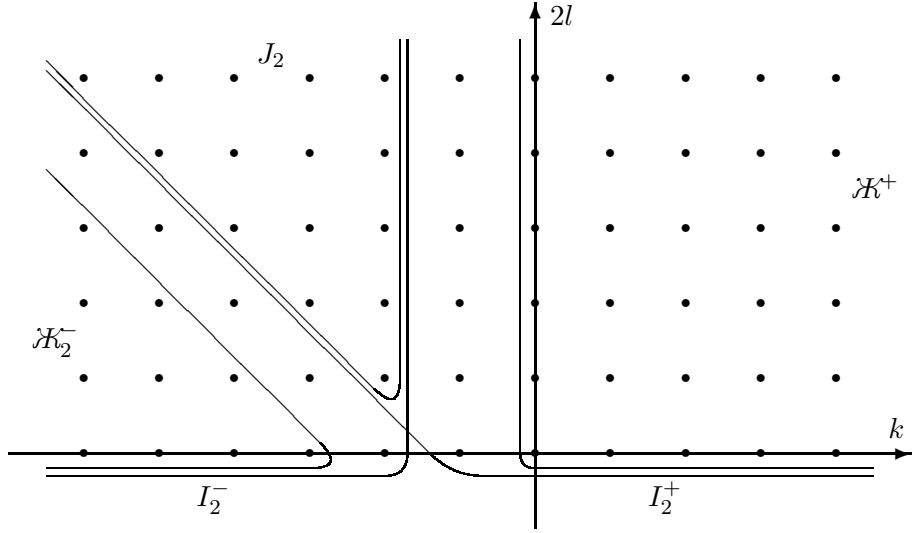
\begin{figure}
\begin{center}
\setlength{\unitlength}{1mm}
\begin{picture}(120,70)
\multiput(10,10)(10,0){11}{\circle*{1}}
\multiput(10,20)(10,0){11}{\circle*{1}}
\multiput(10,30)(10,0){11}{\circle*{1}}
\multiput(10,40)(10,0){11}{\circle*{1}}
\multiput(10,50)(10,0){11}{\circle*{1}}
\multiput(10,60)(10,0){11}{\circle*{1}}

\thicklines
\put(70,0){\vector(0,1){70}}
\put(0,10){\vector(1,0){120}}

\thinlines
\put(68,10){\line(0,1){55}}
\put(70,8){\line(1,0){45}}
\qbezier(68,10)(68,8)(70,8)

\put(52,20){\line(0,1){45}}
\put(48.6,18.6){\line(-1,1){43.6}}
\qbezier(52,20)(52,15.2)(48.6,18.6)

\put(5,8){\line(1,0){35}}
\put(41.4,11.4){\line(-1,1){36.4}}
\qbezier(40,8)(44.8,8)(41.4,11.4)

\put(63,7){\line(1,0){52}}
\put(56,10){\line(-1,1){51}}
\qbezier(63,7)(59,7)(56,10)

\put(5,7){\line(1,0){45}}
\put(53,10){\line(0,1){55}}
\qbezier(50,7)(53,7)(53,10)

\put(72,67){$2l$}
\put(117,12){$k$}
\put(3,24){$\Zh_2^-$}
\put(33,62){$J_2$}
\put(112,44){$\Zh^+$}
\put(25,3){$I_2^-$}
\put(85,3){$I_2^+$}
\end{picture}
\end{center}
\caption{Decomposition of $(\varpi_2^l,\Zh)$ and $(\varpi_2^r,\Zh)$
into irreducible components.}
\label{decomposition-fig2}
\end{figure}

The quotient representations can be identified as follows:

\begin{prop}[Proposition 10 in \cite{L}]  \label{quotient-prop}
As representations of $\mathfrak{gl}(2,\HC)$,
$$
\bigl( \varpi_2^l, \Zh/(I_2^- \oplus \Zh^+) \bigr) \simeq (\pi^0_l, {\cal H}^+),
\qquad 
\bigl( \varpi_2^l, \Zh/(\Zh_2^- \oplus I_2^+) \bigr) \simeq (\pi^0_l, {\cal H}^-),
$$
$$
\bigl( \varpi_2^r, \Zh/(I_2^- \oplus \Zh^+) \bigr) \simeq (\pi^0_r, {\cal H}^+),
\qquad
\bigl( \varpi_2^r, \Zh/(\Zh_2^- \oplus I_2^+) \bigr) \simeq (\pi^0_r, {\cal H}^-),
$$
in all cases the isomorphism map being
$$
{\cal H}^{\pm} \ni \phi(Z) \quad \mapsto \quad
\frac{\degt \phi(Z)}{N(Z)} \in
\begin{matrix} \Zh/(\Zh_2^- \oplus I_2^+) \\ \text{or} \\
\Zh/(I_2^- \oplus \Zh^+). \end{matrix}
$$
The inverse of this isomorphism is given by
\begin{equation}  \label{inverse-iso}
\begin{matrix} \Zh/(\Zh_2^- \oplus I_2^+) \\ \text{or} \\
\Zh/(I_2^- \oplus \Zh^+) \end{matrix} \ni f(Z)
\quad \mapsto \quad \biggl\langle f(Z), \frac1{N(Z-W)} \biggr\rangle_Z
= \frac i{2\pi^3} \int_{Z \in U(2)_R} \frac{f(Z)\,dV}{N(Z-W)} \in {\cal H}.
\end{equation}
\end{prop}

We extend the $\pi_l^0$ and $\pi_r^0$ actions of $GL(2,\HC)$ on
$\widetilde{\cal H}$ to $\widetilde{\Zh}$.
Differentiating these actions, we obtain actions of $\mathfrak{gl}(2,\HC)$,
which preserve $\Zh$, $\Zh^+$ (and, of course, ${\cal H}^-$, ${\cal H}^+$).
These actions are given by the same formulas as in Subsection 3.2 of \cite{FL2}.
Then we have a bilinear pairing between $(\varpi_2^l,\Zh)$ and $(\pi_r^0, \Zh)$
that formally looks the same as (\ref{pairing}):
\begin{equation}  \label{pairing2}
\langle f_1,f_2 \rangle =
\frac i{2\pi^3} \int_{Z \in U(2)_R} f_1(Z) \cdot f_2(Z) \,dV,
\end{equation}
except now the $\mathfrak{gl}(2,\HC)$-actions on the first and second
components are different: $f_1 \in (\varpi_2^l,\Zh)$ and
$f_2 \in (\pi_r^0, \Zh)$.
This bilinear pairing is $\mathfrak{gl}(2,\HC)$-invariant, non-degenerate and
independent of the choice of $R>0$.
In other words, the representations $(\varpi_2^l,\Zh)$ and $(\pi_r^0, \Zh)$
are dual to each other.
Similarly, we have a bilinear pairing between $(\varpi_2^r,\Zh)$ and
$(\pi_l^0, \Zh)$ given by the same formula (\ref{pairing2}).

Now, let us restrict $f_2$ to $(\pi_r^0,{\cal H}) \subset (\pi_r^0, \Zh)$.
Then, by (\ref{orthogonality}), this pairing annihilates all
$f_1 \in (\varpi_2^l, \Zh_2^- \oplus J_2 \oplus \Zh^+)$.
Hence this pairing descends to a pairing between $(\pi_r^0,{\cal H})$ and
$\bigl(\varpi_2^l, \Zh/(\Zh_2^- \oplus J_2 \oplus \Zh^+)\bigr)$.
By Proposition \ref{quotient-prop}, the latter representation is isomorphic
to $(\pi_l^0,{\cal H})$.
Thus we obtain the following expression for a $\mathfrak{gl}(2,\HC)$-invariant
bilinear pairing between $(\pi_l^0,{\cal H})$ and $(\pi_r^0,{\cal H})$:
\begin{equation}  \label{H-pairing2}
(\phi_1,\phi_2) =
\frac i{2\pi^3} \int_{Z \in U(2)_R} (\degt\phi_1)(Z) \cdot \phi_2(Z)
\,\frac{dV}{N(Z)}, \qquad \phi_1,\phi_2 \in {\cal H}.
\end{equation}
(This pairing is independent of the choice of $R>0$.)
Comparing the orthogonality relations (\ref{H-orthogonality}) and
(\ref{orthogonality}), we see that the pairings (\ref{H-pairing})
and (\ref{H-pairing2}) coincide when $\phi_1 \in {\cal H}^+$,
$\phi_2 \in {\cal H}^-$ (but differ for other choices of
$\phi_1$ and $\phi_2$).

\section{Conformal Four-Point Integrals and Magic Identities} \label{fd-section}

In this section we introduce the conformal four-point integrals
$l^{(n)}(Z_1,Z_2;W_1,W_2)$ represented by the $n$-loop box diagrams and explain
the ``magic identities'' due to \cite{DHSS} that assert that integrals
represented by diagrams with the same number of loops are, in fact, equal to
each other. Then we introduce integral operators $L^{(n)}$ on
${\cal H}^+ \otimes {\cal H}^+$ and state the main results of this article.

\subsection{Conformal Four-Point Integrals}

In this subsection we explain how to construct the box diagrams and
the corresponding conformal four-point integrals.

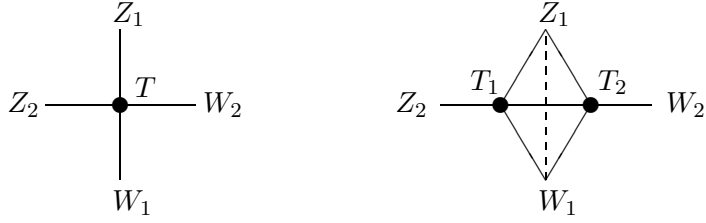
\begin{figure}
\begin{center}
\begin{subfigure}{0.3\textwidth}
\centering
\setlength{\unitlength}{1mm}
\begin{picture}(30,28)
\put(15,14){\circle*{2}}
\put(15,4){\line(0,1){20}}
\put(5,14){\line(1,0){20}}
\put(0,13){$Z_2$}
\put(14,25){$Z_1$}
\put(26,13){$W_2$}
\put(14,0){$W_1$}
\put(17,15){$T$}
\end{picture}
\end{subfigure}
\qquad
\begin{subfigure}{0.3\textwidth}
\centering
\setlength{\unitlength}{1mm}
\begin{picture}(40,28)
\put(14,14){\circle*{2}}
\put(26,14){\circle*{2}}
\put(6,14){\line(1,0){28}}
\put(14,14){\line(3,5){6}}
\put(14,14){\line(3,-5){6}}
\put(26,14){\line(-3,5){6}}
\put(26,14){\line(-3,-5){6}}
\multiput(20,4)(0,2){10}{\line(0,1){1}}
\put(0,13){$Z_2$}
\put(19,25){$Z_1$}
\put(36,13){$W_2$}
\put(19,0){$W_1$}
\put(27,16){$T_2$}
\put(10,16){$T_1$}
\end{picture}
\end{subfigure}
\end{center}
\caption{One-loop (left) and two-loop (right) box (or ladder) diagrams.}
\label{12ladder}
\end{figure}

As in \cite{DHSS}, we use the coordinate space variable notation
(as opposed to the momentum notation).
With this choice of variables, the one- and two-loop box (or ladder) diagrams
are represented as in Figure \ref{12ladder}.
The simplest conformal four-point integral is the one-loop box integral
\begin{equation*}
l^{(1)}(Z_1,Z_2;W_1,W_2) =
\frac i{2\pi^3} \int_{T \in U(2)_r}
\frac{dV}{N(Z_1-T) \cdot N(Z_2-T) \cdot N(W_1-T) \cdot N(W_2-T)}.
\end{equation*}
Here, $r>0$, $Z_1,Z_2 \in \BB D^-_r$ and $W_1,W_2 \in \BB D^+_r$.
Then we have the two-loop box integral
\begin{multline*}
-4\pi^6 \cdot l^{(2)}(Z_1,Z_2;W_1,W_2) \\
= \iint_{\genfrac{}{}{0pt}{}{T_1 \in U(2)_{r_1}}{T_2 \in U(2)_{r_2}}}
\frac{|Z_1-W_1|^2 \cdot |T_1-T_2|^{-2} \, dV_{T_1} \, dV_{T_2}}
{|Z_1-T_1|^2 \cdot |Z_2-T_1|^2 \cdot
|W_1-T_1|^2 \cdot |Z_1-T_2|^2 \cdot |W_1-T_2|^2 \cdot |W_2-T_2|^2},
\end{multline*}
where we write $|Z-W|^2$ for $N(Z-W)$ in order to fit the formula on page.
Here, $r_1>r_2>0$, $Z_1, Z_2 \in \BB D^-_{r_1}$,
$W_1, W_2 \in \BB D^+_{r_2}$.
The factor $|Z_1-W_1|^2=N(Z_1-W_1)$ in the numerator is not involved in
integration and gives $l^{(2)}$ desired conformal properties
(Lemma \ref{conformal}).

In general, one obtains the integral from the box diagram by building
a rational function by writing a factor
$$
\begin{cases}
N(Y_i-Y_j)^{-1} &
\text{if there is a solid edge joining variables $Y_i$ and $Y_j$};  \\
N(Y_i-Y_j) & \text{if there is a dashed edge joining variables $Y_i$ and $Y_j$},
\end{cases}
$$
and then integrating over the solid vertices.
The issue of contours of integration (and, in particular, their relative
position) will be addressed at the end of this subsection.

The box diagrams are obtained by starting with the one-loop box diagram
(Figure \ref{12ladder}) and attaching the so-called ``slingshots'',
as explained in \cite{DHSS}.
Figures \ref{slingshot+one-loop1} and \ref{slingshot+one-loop2}
show the two possible results of attaching a slingshot to the one-loop diagram;
these are called the two-loop box diagrams.
Then Figures \ref{slingshot+two-loop1} and \ref{slingshot+two-loop2}
show two different results of attaching a slingshot to the two-loop box
diagrams; these are called the three-loop box diagrams.
In general, if one has an $(n-1)$-loop box diagram $d^{(n-1)}$ -- that is a
box diagram  obtained by attaching $n-2$ slingshots to the one-loop box diagram
-- there are four ways of attaching a slingshot to form an $n$-loop box diagram
$d^{(n)}$: the hollow vertex of the slingshot can be attached to any of the
vertices labeled $Z_1$, $Z_2$, $W_1$ or $W_2$.
For example, Figure \ref{slingshot+gendiagram}
illustrates a slingshot with the hollow vertex being attached to
the vertex labeled $Z_2$, then the ends of the slingshot with the ``string''
are attached to the adjacent vertices $Z_1$ and $W_1$, the hollow vertex of
the slingshot becomes solid and gets relabeled $T_n$, finally,
the vertex at the tip of the ``handle'' of the slingshot is labeled $Z_2$.
The other three cases are similar.
While there are four ways to attach a slingshot to $d^{(n-1)}$,
some of the resulting diagrams may be the same, since we treat
all slingshots as identical.
(The variables $T_1,\dots,T_n$ get integrated out, so we treat the
diagrams obtained by permuting these variables as the same.)
Thus there are only two two-loop box diagrams, and they differ only by
rearranging labels $Z_1$, $Z_2$, $W_1$, $W_2$
(Figures \ref{slingshot+one-loop1} and \ref{slingshot+one-loop2}).
Figure \ref{ladder-var-label} shows a particular example of an $n$-loop
box diagram called the $n$-loop ladder diagram.
The reason for the ``box'', ``ladder'' and ``loop'' terminology becomes
apparent when one switches to the momentum variables, see Figure \ref{ladders},
and more figures are given is \cite{DHSS}.

\begin{figure}
\begin{center}
\begin{subfigure}{0.15\textwidth}
\centering
\includegraphics[scale=1]{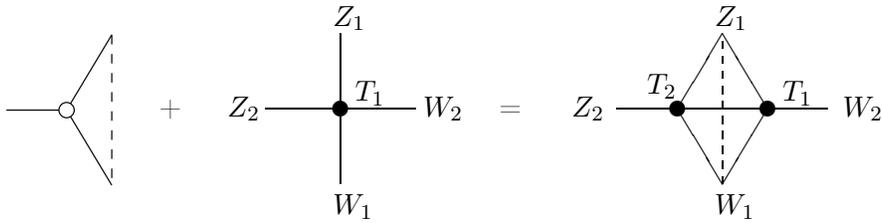}
\end{subfigure}
+
\begin{subfigure}{0.25\textwidth}
\centering
\setlength{\unitlength}{1mm}
\begin{picture}(30,28)
\put(15,14){\circle*{2}}
\put(15,4){\line(0,1){20}}
\put(5,14){\line(1,0){20}}
\put(0,13){$Z_2$}
\put(14,25){$Z_1$}
\put(26,13){$W_2$}
\put(14,0){$W_1$}
\put(17,15){$T_1$}
\end{picture}
\end{subfigure}
=
\begin{subfigure}{0.32\textwidth}
\centering
\setlength{\unitlength}{1mm}
\begin{picture}(40,28)
\put(14,14){\circle*{2}}
\put(26,14){\circle*{2}}
\put(6,14){\line(1,0){28}}
\put(14,14){\line(3,5){6}}
\put(14,14){\line(3,-5){6}}
\put(26,14){\line(-3,5){6}}
\put(26,14){\line(-3,-5){6}}
\multiput(20,4)(0,2){10}{\line(0,1){1}}
\put(0,13){$Z_2$}
\put(19,25){$Z_1$}
\put(36,13){$W_2$}
\put(19,0){$W_1$}
\put(28,15){$T_1$}
\put(10,16){$T_2$}
\end{picture}
\end{subfigure}
\end{center}
\caption{Attaching a slingshot to the one-loop box diagram.}
\label{slingshot+one-loop1}
\end{figure}

\begin{figure}
\begin{center}
\begin{subfigure}{0.15\textwidth}
\centering
\includegraphics[scale=1]{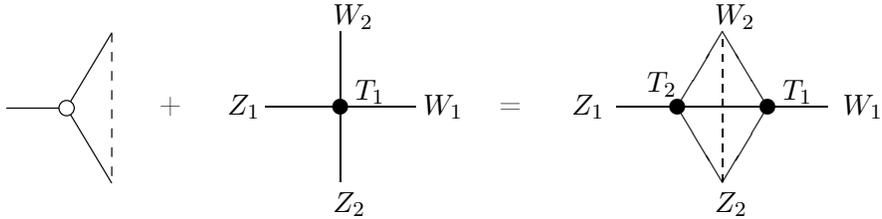}
\end{subfigure}
+
\begin{subfigure}{0.25\textwidth}
\centering
\setlength{\unitlength}{1mm}
\begin{picture}(30,28)
\put(15,14){\circle*{2}}
\put(15,4){\line(0,1){20}}
\put(5,14){\line(1,0){20}}
\put(0,13){$Z_1$}
\put(14,25){$W_2$}
\put(26,13){$W_1$}
\put(14,0){$Z_2$}
\put(17,15){$T_1$}
\end{picture}
\end{subfigure}
=
\begin{subfigure}{0.32\textwidth}
\centering
\setlength{\unitlength}{1mm}
\begin{picture}(40,28)
\put(14,14){\circle*{2}}
\put(26,14){\circle*{2}}
\put(6,14){\line(1,0){28}}
\put(14,14){\line(3,5){6}}
\put(14,14){\line(3,-5){6}}
\put(26,14){\line(-3,5){6}}
\put(26,14){\line(-3,-5){6}}
\multiput(20,4)(0,2){10}{\line(0,1){1}}
\put(0,13){$Z_1$}
\put(19,25){$W_2$}
\put(36,13){$W_1$}
\put(19,0){$Z_2$}
\put(28,15){$T_1$}
\put(10,16){$T_2$}
\end{picture}
\end{subfigure}
\end{center}
\caption{Another way of attaching a slingshot to the one-loop box diagram.}
\label{slingshot+one-loop2}
\end{figure}

\begin{figure}
\begin{center}
\begin{subfigure}{0.15\textwidth}
\centering
\includegraphics[scale=1]{slingshot.eps}
\end{subfigure}
+
\begin{subfigure}{0.3\textwidth}
\centering
\setlength{\unitlength}{1mm}
\begin{picture}(40,28)
\put(14,14){\circle*{2}}
\put(26,14){\circle*{2}}
\put(6,14){\line(1,0){28}}
\put(14,14){\line(3,5){6}}
\put(14,14){\line(3,-5){6}}
\put(26,14){\line(-3,5){6}}
\put(26,14){\line(-3,-5){6}}
\multiput(20,4)(0,2){10}{\line(0,1){1}}
\put(0,13){$Z_2$}
\put(19,25){$Z_1$}
\put(36,13){$W_2$}
\put(19,0){$W_1$}
\end{picture}
\end{subfigure}
=
\begin{subfigure}{0.32\textwidth}
\centering
\includegraphics[scale=1]{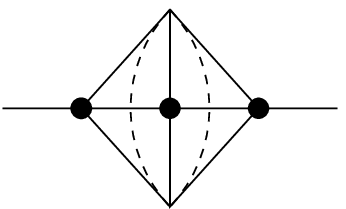}
\end{subfigure}
\hspace{-0.33\textwidth}
\begin{subfigure}{0.32\textwidth}
\centering
\setlength{\unitlength}{1mm}
\begin{picture}(46,28)
\put(0,13){$Z_2$}
\put(21,25){$Z_1$}
\put(41,13){$W_2$}
\put(21,0){$W_1$}
\end{picture}
\end{subfigure}
\end{center}
\caption{Attaching a slingshot to a two-loop box diagram.}
\label{slingshot+two-loop1}
\end{figure}

\begin{figure}
\begin{center}
\begin{subfigure}{0.15\textwidth}
\centering
\includegraphics[scale=1]{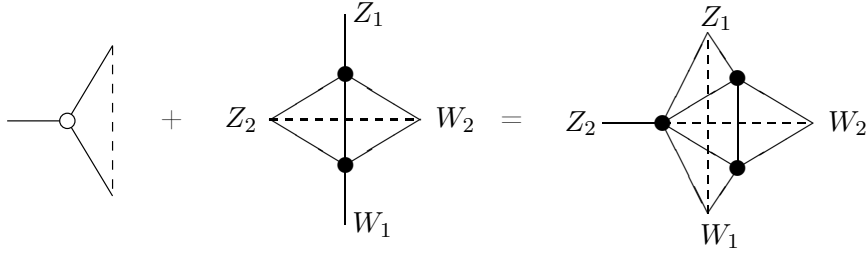}
\end{subfigure}
+
\begin{subfigure}{0.25\textwidth}
\centering
\setlength{\unitlength}{1mm}
\begin{picture}(31,30)
\put(16,9){\circle*{2}}
\put(16,21){\circle*{2}}
\put(16,1){\line(0,1){28}}
\put(6,15){\line(5,3){10}}
\put(6,15){\line(5,-3){10}}
\put(26,15){\line(-5,3){10}}
\put(26,15){\line(-5,-3){10}}
\multiput(6,15)(2,0){10}{\line(1,0){1}}
\put(0,14){$Z_2$}
\put(17,28){$Z_1$}
\put(28,14){$W_2$}
\put(17,0){$W_1$}
\end{picture}
\end{subfigure}
=
\begin{subfigure}{0.3\textwidth}
\centering
\setlength{\unitlength}{1mm}
\begin{picture}(39,33)
\put(23,10){\circle*{2}}
\put(23,22){\circle*{2}}
\put(13,16){\circle*{2}}
\put(5,16){\line(1,0){8}}
\put(23,10){\line(0,1){12}}
\put(13,16){\line(1,2){6}}
\put(23,22){\line(-2,3){4}}
\put(13,16){\line(1,-2){6}}
\put(23,10){\line(-2,-3){4}}
\put(13,16){\line(5,3){10}}
\put(13,16){\line(5,-3){10}}
\put(33,16){\line(-5,3){10}}
\put(33,16){\line(-5,-3){10}}
\multiput(13,16)(2,0){10}{\line(1,0){1}}
\multiput(19,4)(0,2){12}{\line(0,1){1}}
\put(0,15){$Z_2$}
\put(18,29){$Z_1$}
\put(35,15){$W_2$}
\put(18,0){$W_1$}
\end{picture}
\end{subfigure}
\end{center}
\caption{Another way of attaching a slingshot to a two-loop box diagram.}
\label{slingshot+two-loop2}
\end{figure}

\begin{figure}
\begin{center}
\begin{subfigure}{0.15\textwidth}
\centering
\includegraphics[scale=1]{slingshot.eps}
\end{subfigure}
+
\begin{subfigure}{0.35\textwidth}
\centering
\includegraphics[scale=1]{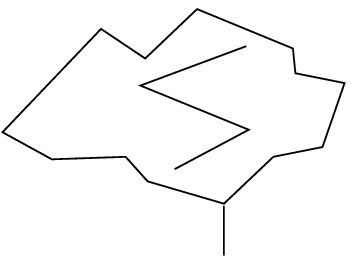}
\end{subfigure}
=
\begin{subfigure}{0.4\textwidth}
\centering
\includegraphics[scale=1]{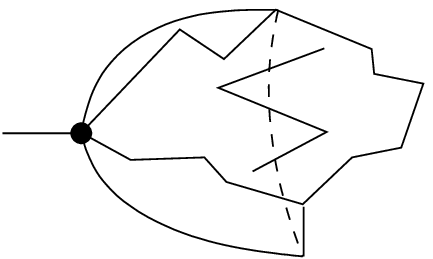}
\end{subfigure}
\end{center}
\vskip-33mm \hskip62mm $Z_1$ \hskip63.5mm $Z_1$

\vskip4mm \hskip77mm $W_2$ \hskip62.5mm $W_2$

\vskip-2mm \hskip106mm $T_n$

\vskip-2mm \hskip37mm $Z_2$ \hskip56mm $Z_2$

\vskip8mm \hskip65.5mm $W_1$ \hskip62.5mm $W_1$
\caption{Attaching a slingshot to a general box diagram.}
\label{slingshot+gendiagram}
\end{figure}

\begin{figure}
\begin{center}
\begin{subfigure}{0.2\textwidth}
\centering
\includegraphics[scale=.8]{Feynman2.eps}
\end{subfigure}
\begin{subfigure}{0.25\textwidth}
\centering
\includegraphics[scale=.8]{Feynman3.eps}
\end{subfigure}
\begin{subfigure}{0.42\textwidth}
\centering
\includegraphics[scale=.8]{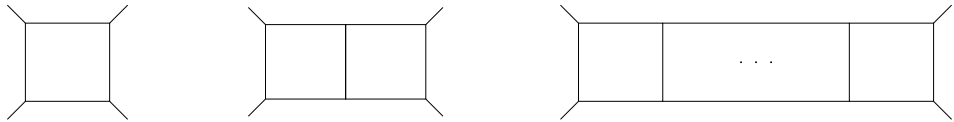}
\end{subfigure}
\end{center}
\caption{One-, two- and $n$-loop box or ladder diagrams in momentum variables.}
\label{ladders}
\end{figure}

In order to specify the cycles of integration, we introduce a partial ordering
on the variables in each $n$-loop box diagram $d^{(n)}$.
For the one-loop box diagram (Figure \ref{12ladder}) the relations are
$$
W_1,W_2 \prec T \prec Z_1,Z_2.
$$
Suppose that an $n$-loop box diagram $d^{(n)}$ is obtained
from an $(n-1)$-loop diagram $d^{(n-1)}$ by adding a slingshot.
Then $d^{(n)}$ will have one new relation for each solid edge of the slingshot,
plus those implied by the transitivity property.
Suppose, by induction, that the partial ordering for the variables in
$d^{(n-1)}$ are already specified.
We label the solid vertices in $d^{(n-1)}$ as $T_1, \dots, T_{n-1}$.
There are exactly four ways of attaching a slingshot to $d^{(n-1)}$
-- so that one of $Z_1$, $Z_2$, $W_1$ or $W_2$ becomes a solid vertex and gets
relabeled as $T_n$.
\begin{itemize}
\item
If $d^{(n)}$ is obtained from $d^{(n-1)}$ by adding the slingshot so
that $Z_1$ becomes a solid vertex, the relations in $d^{(n-1)}$ carry over
to $d^{(n)}$ with $Z_1$ replaced with $T_n$. Then we get new relations
$$
W_2 \prec T_n \prec Z_1, Z_2
$$
(plus those implied by the transitivity property).

\item
If $d^{(n)}$ is obtained from $d^{(n-1)}$ by adding the slingshot so
that $Z_2$ becomes a solid vertex, the relations in $d^{(n-1)}$ carry over
to $d^{(n)}$ with $Z_2$ replaced with $T_n$. Then we get new relations
$$
W_1 \prec T_n \prec Z_1, Z_2
$$
(plus those implied by the transitivity property).

\item
If $d^{(n)}$ is obtained from $d^{(n-1)}$ by adding the slingshot so
that $W_1$ becomes a solid vertex, the relations in $d^{(n-1)}$ carry over
to $d^{(n)}$ with $W_1$ replaced with $T_n$. Then we get new relations
$$
W_1, W_2 \prec T_n \prec Z_2
$$
(plus those implied by the transitivity property).

\item
If $d^{(n)}$ is obtained from $d^{(n-1)}$ by adding the slingshot so
that $W_2$ becomes a solid vertex, the relations in $d^{(n-1)}$ carry over
to $d^{(n)}$ with $W_2$ replaced with $T_n$. Then we get new relations
$$
W_1, W_2 \prec T_n \prec Z_1
$$
(plus those implied by the transitivity property).
\end{itemize}
This completely defines the partial ordering on the variables in $d^{(n)}$.
We choose real numbers $r_1,\dots,r_n>0$ such that $r_i < r_j$
whenever $T_i \prec T_j$ (it is easy to check that such a choice is
always possible). Then each $T_k$ gets integrated over $U(2)_{r_k}$.
Finally,
\begin{equation}  \label{r_max}
Z_i \in \BB D^-_{r_{\text{max},i}}, \quad \text{where }
r_{\text{max},i} = \max\{r_k ;\: T_k \prec Z_i \}, \qquad i=1,2;
\end{equation}
\begin{equation}  \label{r_min}
W_i \in \BB D^+_{r_{\text{min},i}}, \quad \text{where }
r_{\text{min},i} = \min\{r_k ;\: W_i \prec T_k \}, \qquad i=1,2.
\end{equation}

If desired, by Corollary 90 in \cite{FL1} the integrals over various $U(2)_r$'s
can be replaced by integrals over the Minkowski space $\BB M$ via an
appropriate ``Cayley transform''.
This means that these integrals are what the physicists call
``the off-shell Minkowski integrals''.

\subsection{Magic Identities}  \label{magic-id-subsect}

In this subsection we state the so-called ``magic identities'' due to
J.~M.~Drummond, J.~Henn, V.~A.~Smirnov and E.~Sokatchev \cite{DHSS}.
Informally, they assert that all conformal four-point box integrals obtained
by adding the same number of slingshots to the one-loop integral are equal.
In other words, only the number of slingshots matters and
not how they are attached.

\begin{thm}  \label{magic}
Let $l^{(n)}(Z_1,Z_2;W_1,W_2)$ and $\tilde l^{(n)}(Z_1,Z_2;W_1,W_2)$ be two
conformal four-point integrals corresponding to any two $n$-loop box diagrams,
then
$$
l^{(n)}(Z_1,Z_2;W_1,W_2) = \tilde l^{(n)}(Z_1,Z_2;W_1,W_2).
$$
\end{thm}

In particular, we can parametrize the conformal four-point integrals by the
number of loops in the diagrams and choose a single representative from the set
of all $n$-loop diagrams, such as the $n$-loop ladder diagram
(Figures \ref{ladders} and \ref{ladder-var-label}).

The original paper \cite{DHSS} gives a proof for the Euclidean metric case and
claims that the result is also true for the Minkowski metric case.
In the Euclidean case, the box integrals are produced by making all
variables belong to $\BB H$ and replacing all cycles of integration by $\BB H$.
Then $N(X-Y)$ is just the square of the Euclidean distance between $X$ and $Y$.
There are no convergence issues whatsoever.
On the other hand, the Minkowski case (which is the case we consider)
is much more subtle. In order to deal with convergence issues, we must
consider the so-called ``off-shell Minkowski integrals'' or make the cycles
of integration to be various $U(2)_r$'s.
Then the relative position of cycles becomes very important.
As can be seen in the course of proof of Theorem \ref{main-thm}, choosing
the ``wrong'' cycles typically results in integral being zero.

The proof of Theorem \ref{magic} given in \cite{DHSS} can be outlined as
follows. First, they prove a symmetry relationship for the two-loop integrals
represented by the two-loop diagram in Figure \ref{12ladder}
$$
l^{(2)}(Z_1,Z_2;W_1,W_2) = l^{(2)}(Z_2,Z_1;W_2,W_1);
$$
this is done by direct computation.
Then they prove the magic identity for the integrals represented by
the three-loop diagrams in Figures \ref{slingshot+two-loop1} and
\ref{slingshot+two-loop2}
$$
l^{(3)}(Z_1,Z_2;W_1,W_2) = \tilde l^{(3)}(Z_1,Z_2;W_1,W_2);
$$
this is also done by direct computation.
These identities can be represented by the box diagrams,
as shown in Figure \ref{magic-pf}.
Finally, they apply induction on the number of loops or slingshots.

\begin{figure}
\begin{center}
\begin{subfigure}{0.23\textwidth}
\centering
\setlength{\unitlength}{.75mm}
\begin{picture}(40,28)
\put(14,14){\circle*{2}}
\put(26,14){\circle*{2}}
\put(6,14){\line(1,0){28}}
\put(14,14){\line(3,5){6}}
\put(14,14){\line(3,-5){6}}
\put(26,14){\line(-3,5){6}}
\put(26,14){\line(-3,-5){6}}
\multiput(20,4)(0,2){10}{\line(0,1){1}}
\put(0,13){\footnotesize $Z_2$}
\put(19,25){\footnotesize $Z_1$}
\put(36,13){\footnotesize $W_2$}
\put(19,0){\footnotesize $W_1$}
\end{picture}
\end{subfigure}
=
\begin{subfigure}{0.16\textwidth}
\centering
\setlength{\unitlength}{.75mm}
\begin{picture}(31,30)
\put(16,9){\circle*{2}}
\put(16,21){\circle*{2}}
\put(16,1){\line(0,1){28}}
\put(6,15){\line(5,3){10}}
\put(6,15){\line(5,-3){10}}
\put(26,15){\line(-5,3){10}}
\put(26,15){\line(-5,-3){10}}
\multiput(6,15)(2,0){10}{\line(1,0){1}}
\put(0,14){\footnotesize $Z_2$}
\put(17,28){\footnotesize $Z_1$}
\put(28,14){\footnotesize $W_2$}
\put(17,0){\footnotesize $W_1$}
\end{picture}
\end{subfigure}
, \quad
\begin{subfigure}{0.25\textwidth}
\centering
\includegraphics[scale=.75]{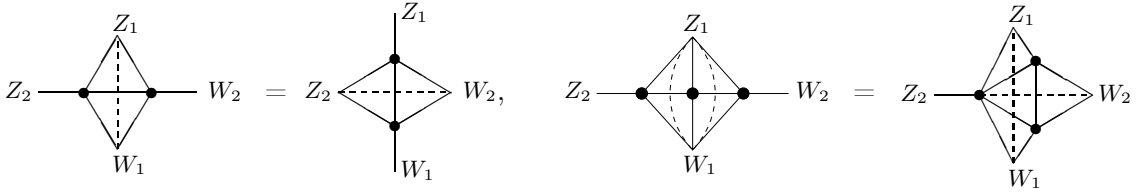}
\end{subfigure}
\hspace{-0.265\textwidth}
\begin{subfigure}{0.25\textwidth}
\centering
\setlength{\unitlength}{.75mm}
\begin{picture}(46,28)
\put(0,13){\footnotesize $Z_2$}
\put(21,25){\footnotesize $Z_1$}
\put(41,13){\footnotesize $W_2$}
\put(21,0){\footnotesize $W_1$}
\end{picture}
\end{subfigure}
=
\begin{subfigure}{0.22\textwidth}
\centering
\setlength{\unitlength}{.75mm}
\begin{picture}(39,33)
\put(23,10){\circle*{2}}
\put(23,22){\circle*{2}}
\put(13,16){\circle*{2}}
\put(5,16){\line(1,0){8}}
\put(23,10){\line(0,1){12}}
\put(13,16){\line(1,2){6}}
\put(23,22){\line(-2,3){4}}
\put(13,16){\line(1,-2){6}}
\put(23,10){\line(-2,-3){4}}
\put(13,16){\line(5,3){10}}
\put(13,16){\line(5,-3){10}}
\put(33,16){\line(-5,3){10}}
\put(33,16){\line(-5,-3){10}}
\multiput(13,16)(2,0){10}{\line(1,0){1}}
\multiput(19,4)(0,2){12}{\line(0,1){1}}
\put(-1,15){\footnotesize $Z_2$}
\put(18,29){\footnotesize $Z_1$}
\put(34,15){\footnotesize $W_2$}
\put(18,0){\footnotesize $W_1$}
\end{picture}
\end{subfigure}
\end{center}
\caption{Ingredients of the proof of the magic identities given in \cite{DHSS}.}
\label{magic-pf}
\end{figure}

\subsection{Statement of the Main Result}

Using the bilinear pairing (\ref{H-pairing2}), we obtain integral operators
$L^{(n)}$ on $(\pi^0_l, {\cal H}^+) \otimes (\pi^0_r, {\cal H}^+)$
that have  the conformal integrals $l^{(n)}$ as their kernels:
\begin{multline*}
L^{(n)} (\phi_1 \otimes \phi_2)(W_1,W_2) \\
= \Bigl(\frac{i}{2\pi^3}\Bigr)^2
\iint_{\genfrac{}{}{0pt}{}{Z_1 \in U(2)_{R_1}}{Z_2 \in U(2)_{R_2}}}
l^{(n)}(Z_1,Z_2;W_1,W_2) \cdot (\degt_{Z_1} \phi_1)(Z_1)
\cdot (\degt_{Z_2} \phi_2)(Z_2) \,\frac{dV_1}{N(Z_1)} \frac{dV_2}{N(Z_2)},
\end{multline*}
where $\phi_1, \phi_2 \in {\cal H}^+$, $R_1 > r_{\text{max},1}$,
$R_2 > r_{\text{max},2}$, $W_1 \in \BB D^+_{r_{\text{min},1}}$,
$W_2 \in \BB D^+_{r_{\text{min},2}}$
(recall that $r_{\text{max},i}$ and $r_{\text{min},i}$ are defined in
(\ref{r_max}) and (\ref{r_min})).
First, we state a preliminary version of the main result.

\begin{prop}
For each $\phi_1, \phi_2 \in {\cal H}^+$, the function
$L^{(n)} (\phi_1 \otimes \phi_2)(W_1,W_2)$ is polynomial and harmonic
in each variable. In other words,
$L^{(n)} (\phi_1 \otimes \phi_2)(W_1,W_2) \in {\cal H}^+ \otimes {\cal H}^+$.
Moreover, the operator
$$
L^{(n)}: (\pi^0_l, {\cal H}^+) \otimes (\pi^0_r, {\cal H}^+)
\to (\pi^0_l, {\cal H}^+) \otimes (\pi^0_r, {\cal H}^+)
$$
is $\mathfrak{gl}(2,\HC)$-equivariant.
\end{prop}

Our goal is to compute the actions of these integral operators 
$L^{(n)}$ on $(\pi^0_l, {\cal H}^+) \otimes (\pi^0_r, {\cal H}^+)$
and to prove the magic identities for $L^{(n)}$.

The decomposition of $(\pi^0_l, {\cal H}^+) \otimes (\pi^0_r, {\cal H}^+)$
into irreducible components is well known.
This was done in a greater generality, for example, in \cite{JV}.
We provide a summary of this result following \cite{FL1,L}.
Let $k=1,2,3,\dots$, and denote by $\BB C^{k \times k}$ the space of complex
$k \times k$ matrices.
Then $\widetilde{\Zh} \otimes \BB C^{k \times k}$ can be thought of as
the space of holomorphic functions on $\HC$ (possibly with singularities)
with values in $\BB C^{k \times k}$.
Recall the actions $\rho_k$ of $GL(2,\HC)$ on
$\widetilde{\Zh} \otimes \BB C^{k \times k}$ described by
equation (60) in \cite{FL1}:
\begin{equation}  \label{rho_k-action}
\rho_k(h): \: F(Z) \quad \mapsto \quad \bigl( \rho_k(h)F \bigr)(Z) =
\frac {\tau_{\frac{k-1}2}(cZ+d)^{-1}}{N(cZ+d)} \cdot
F \bigl( (aZ+b)(cZ+d)^{-1} \bigr) \cdot
\frac {\tau_{\frac{k-1}2}(a'-Zc')^{-1}}{N(a'-Zc')},
\end{equation}
where
$h = \bigl(\begin{smallmatrix} a' & b' \\ c' & d' \end{smallmatrix}\bigr)
\in GL(2,\HC)$,
$h^{-1} = \bigl(\begin{smallmatrix} a & b \\ c & d \end{smallmatrix}\bigr)$,
expressions $cZ+d$ and $a'-Zc'$ are regarded as elements of $\HC^{\times}$,
and $\tau_l: \HC^{\times} \to\operatorname{Aut}(\BB C^{2l+1}) \subset
\BB C^{(2l+1) \times (2l+1)}$
is the irreducible $(2l+1)$-dimensional representation of $\HC^{\times}$
described in Subsection \ref{matrix-coeff-subsection},
$l=0,\frac12,1,\frac32, \dots$.

Differentiating this action, we obtain an action of
$\mathfrak{gl}(2,\HC)$ which preserves $\Zh \otimes \BB C^{k \times k}$
and $\Zh^+ \otimes \BB C^{k \times k}$.
As a special case of Proposition 4.7 in \cite{JV}
(see also the discussion preceding the proposition and references therein),
we have:

\begin{thm} \label{JV-thm}
The representations $(\rho_k, \Zh^+ \otimes \BB C^{k \times k})$,
$k=1,2,3,\dots$, of $\mathfrak{sl}(2,\HC)$ are irreducible,
pairwise non-isomorphic.
They possess inner products which make them unitary representations of
the real form $\mathfrak{su}(2,2)$ of $\mathfrak{sl}(2,\HC)$.
\end{thm}

According to \cite{JV}, we have the following decomposition of the tensor
product $(\pi^0_l, {\cal H}^+) \otimes (\pi^0_r, {\cal H}^+)$ into irreducible
subrepresentations of $\mathfrak{gl}(2,\HC)$:
\begin{equation}  \label{tensor-decomp}
(\pi^0_l, {\cal H}^+) \otimes (\pi^0_r, {\cal H}^+) \simeq
\bigoplus_{k=1}^{\infty} (\rho_k,\Zh^+\otimes \BB C^{k \times k})
\end{equation}
(see also Subsection 5.1 in \cite{FL1}).
We outline the proof of this statement.
First of all, by Lemma 10 in \cite{FL1}, the tensor product
$(\pi^0_l, {\cal H}^+) \otimes (\pi^0_r, {\cal H}^+)$ contains each
$(\rho_k,\Zh^+\otimes \BB C^{k \times k})$ with 
$$
(\rho_1,\Zh^+) \quad \text{generated by} \quad 1 \otimes 1
$$
and
$$
(\rho_k,\Zh^+\otimes \BB C^{k \times k}) \quad \text{generated by} \quad
(z_{ij}-z'_{ij})^{k-1} , \qquad k \ge 2.
$$
Then one checks that the direct sum
$\bigoplus_{k=1}^{\infty} (\rho_k,\Zh^+\otimes \BB C^{k \times k})$
exhausts all of $(\pi^0_l, {\cal H}^+) \otimes (\pi^0_r, {\cal H}^+)$
by comparing the two sides as representations of $U(2) \times U(2)$ or
$\mathfrak{u}(2) \times \mathfrak{u}(2)$.

In order to state the full version of the main result, we introduce
coefficients $a^k(n,p)$, where $k=0,1,2,\dots$ and $0 \le p \le k$,
that are defined by the following recursive relations:
\begin{equation}  \label{a-ini}
a^k(1,p) = \frac1{k+1}, \qquad p=0,1,\dots,k,
\end{equation}
and
\begin{equation}  \label{a-rec-rel}
a^k(n+1,p) = \sum_{q=p}^k \frac1{q+1} \cdot a^k(n,q).
\end{equation}

\begin{thm}  \label{main-thm}
The operator $L^{(n)}$ associated to any $n$-loop box diagram
maps ${\cal H}^+ \otimes {\cal H}^+$ into ${\cal H}^+ \otimes {\cal H}^+$,
and the map
\begin{equation}  \label{Ln}
L^{(n)} : (\pi^0_l, {\cal H}^+) \otimes (\pi^0_r, {\cal H}^+) \to
(\pi^0_l, {\cal H}^+) \otimes (\pi^0_r, {\cal H}^+)
\end{equation}
is $\mathfrak{gl}(2,\HC)$-equivariant.
If $x \in (\pi^0_l, {\cal H}^+) \otimes (\pi^0_r, {\cal H}^+)$
belongs to an irreducible component isomorphic to
$(\rho_k,\Zh^+ \otimes \BB C^{k \times k})$
in the decomposition (\ref{tensor-decomp}), then
$$
L^{(n)}(x) = \mu^{(n)}_k x, \qquad \text{where} \qquad
\mu^{(n)}_k = \sum_{p=0}^{k-1} (-1)^{k+p+1} \cdot a^{k-1}(n,p) \cdot
\begin{pmatrix} k-1 \\ p \end{pmatrix}.
$$
\end{thm}

In particular, we obtain the magic identities for operators $L^{(n)}$:

\begin{cor}  \label{magic-id-operators}
Let $L^{(n)}$ and $\tilde L^{(n)}$ be two integral operators
corresponding to any two $n$-loop box diagrams, then
$L^{(n)} = \tilde L^{(n)}$, as operators on ${\cal H}^+ \otimes {\cal H}^+$.
\end{cor}

\begin{rem}
If one can prove that each $n$-loop box integral $l^{(n)}(Z_1,Z_2;W_1,W_2)$ is
harmonic in each variable $Z_1$, $Z_2$, $W_1$ and $W_2$,
then it is easy to show that Theorem \ref{main-thm}
implies the magic identities, as stated in Theorem \ref{magic}.
\end{rem}

\subsection{Example: The Case $n=2$}

In this subsection we compute the coefficients $\mu^{(2)}_k$ and show that
Theorem 26 in \cite{L} is a special case of Theorem \ref{main-thm} when $n=2$.
We have:
$$
a^k(2,p) = \sum_{q=p}^k \frac1{q+1} \cdot a^k(1,q)
= \frac1{k+1} \sum_{q=p}^k \frac1{q+1}
$$
and
\begin{multline*}
\mu^{(2)}_k = \sum_{p=0}^{k-1} (-1)^{k+p+1} \cdot a^{k-1}(2,p) \cdot
\begin{pmatrix} k-1 \\ p \end{pmatrix}  \\
= \frac1k \sum_{p=0}^{k-1} (-1)^{k+p+1} \cdot
\begin{pmatrix} k-1 \\ p \end{pmatrix} \sum_{q=p}^{k-1} \frac1{q+1}
= \frac{(-1)^{k+1}}k \sum_{q=0}^{k-1} \frac1{q+1} \sum_{p=0}^q (-1)^p
\begin{pmatrix} k-1 \\ p \end{pmatrix}.
\end{multline*}
If $k=1$, we obtain $\mu^{(2)}_1 =1$. So, assume $k \ge 2$.
Using an identity
$$
\sum_{p=0}^q (-1)^p \begin{pmatrix} k \\ p \end{pmatrix}
= (-1)^q \begin{pmatrix} k-1 \\ q \end{pmatrix}
$$
which can be easily proved by induction (see formula 0.15(4) in \cite{GR}),
we obtain
$$
\mu^{(2)}_k = \frac{(-1)^{k+1}}k \sum_{q=0}^{k-1} \frac{(-1)^q}{q+1}
\begin{pmatrix} k-2 \\ q \end{pmatrix}  \\
= \frac{(-1)^{k+1}}{k(k-1)} \sum_{q=0}^{k-1} (-1)^q \cdot
\begin{pmatrix} k-1 \\ q+1 \end{pmatrix}
= \frac{(-1)^{k+1}}{k(k-1)}.
$$
This shows that
$$
\mu^{(2)}_k =
\begin{cases}
1 & \text{if $k=1$;} \\
\frac{(-1)^{k+1}}{k(k-1)} & \text{if $k \ge 2$;}
\end{cases}
$$
and that Theorem 26 in \cite{L} is a special case of Theorem \ref{main-thm}.

\section{Proof of Theorem \ref{main-thm}}  \label{proof-section}

\subsection{Preliminary Lemmas}

In this subsection we prove two lemmas that are part of our
proof of Theorem \ref{main-thm}.
The first lemma describes an important conformal property of
four-point box integrals.

\begin{lem}  \label{conformal}
For each $h = \bigl(\begin{smallmatrix} a' & b' \\
c' & d' \end{smallmatrix}\bigr) \in GL(2,\HC)$
sufficiently close to the identity, we have:
\begin{multline*}
l^{(n)}(\tilde Z_1, \tilde Z_2; \tilde W_1, \tilde W_2) \\
= N(a'-Z_1c') \cdot N(cZ_2+d) \cdot N(cW_1+d) \cdot N(a'-W_2c') \cdot
l^{(n)}(Z_1,Z_2;W_1,W_2),
\end{multline*}
where $h^{-1} = \bigl(\begin{smallmatrix} a & b \\
c & d \end{smallmatrix}\bigr)$,
$\tilde Z_i = (aZ_i+b)(cZ_i+d)^{-1}$ and $\tilde W_i = (aW_i+b)(cW_i+d)^{-1}$,
$i=1,2$.
\end{lem}

\begin{proof}
The proof is by induction on $n$;
for $n=1$ and $n=2$ this is Lemma 14 in \cite{L}.
For concreteness, let us assume that the last slingshot is attached to an
$(n-1)$-loop box diagram $d^{(n-1)}$ so that $Z_1$ becomes a solid vertex and
gets relabeled as $T_n$ (the other cases are similar). Then
$$
l^{(n)}(Z_1,Z_2;W_1,W_2) = \frac{i}{2\pi^3} \int_{T_n \in U(2)_{r_n}}
\frac{N(Z_2-W_2) \cdot l^{(n-1)}(T_n,Z_2;W_1,W_2)}
{N(Z_1-T_n) \cdot N(Z_2-T_n) \cdot N(W_2-T_n)} \,dV_{T_n},
$$
where $l^{(n-1)}(Z_1,Z_2;W_1,W_2)$ is the conformal four-point integral
corresponding to the $(n-1)$-loop diagram $d^{(n-1)}$.
By induction, we assume that the result holds for $l^{(n-1)}(Z_1,Z_2;W_1,W_2)$.
Using Lemmas 10, 61 from \cite{FL1} and letting
$\tilde T_n = (aT_n+b)(cT_n+d)^{-1}$, we obtain
\begin{multline*}
l^{(n)}(\tilde Z_1, \tilde Z_2; \tilde W_1, \tilde W_2)
= \frac{i}{2\pi^3} \int_{\tilde T_n \in U(2)_{r_n}}
\frac{N(\tilde Z_2 - \tilde W_2) \cdot
l^{(n-1)}(\tilde T_n,\tilde Z_2;\tilde W_1,\tilde W_2)}
{N(\tilde Z_1 - \tilde T_n) \cdot N(\tilde Z_2 - \tilde T_n)
\cdot N(\tilde W_2 -\tilde T_n)} \,dV_{\tilde T_n} \\
= N(a'-Z_1c') \cdot N(cZ_2+d) \cdot N(cW_1+d) \cdot N(a'-W_2c') \\
\times \frac{i}{2\pi^3} \int_{\tilde T_n \in U(2)_{r_n}}
\frac{N(Z_2-W_2) \cdot N(cT_n+d)^2 \cdot N(a'-T_nc')^2}
{N(Z_1-T_n) \cdot N(Z_2-T_n) \cdot N(W_2-T_n)}
\cdot l^{(n-1)}(T_n,Z_2;W_1,W_2) \,dV_{\tilde T_n}  \\
= N(a'-Z_1c') \cdot N(cZ_2+d) \cdot N(cW_1+d) \cdot N(a'-W_2c') \cdot
l^{(n)}(Z_1,Z_2;W_1,W_2),
\end{multline*}
where we are allowed to replace integration over $\tilde T_n \in U(2)_{r_n}$
with $T_n \in U(2)_{r_n}$ since the integrand is a closed differential form
and $h$ is sufficiently close to the identity.
\end{proof}

The second lemma concerns a set of generators of
${\cal H}^+ \otimes {\cal H}^+$.

\begin{lem}  \label{product-lemma}
As a representation of $\mathfrak{gl}(2,\HC)$,
$(\pi^0_l \otimes \pi^0_r, {\cal H}^+ \otimes {\cal H}^+)$
is generated by $1 \otimes (z'_{11})^k$, $k = 0, 1, 2, \dots$.
It can also be generated by $(z_{11})^k \otimes 1$, $k = 0, 1, 2, \dots$.
\end{lem}

\begin{proof}
Recall that $(z_{11}-z'_{11})^k$ generates the irreducible component
$(\rho_{k+1},\Zh^+\otimes \BB C^{(k+1) \times (k+1)})$ of
${\cal H}^+ \otimes {\cal H}^+$.
We compute the inner product in ${\cal H}^+ \otimes {\cal H}^+$ induced
by (\ref{inner-prod}):
$$
\bigl\langle 1 \otimes (z'_{11})^k, (z_{11}-z'_{11})^k \bigr\rangle_{inn.\: prod.}
= (-1)^k \bigl\langle 1 \otimes (z'_{11})^k, 1 \otimes (z'_{11})^k
\bigr\rangle_{inn.\: prod.} = (-1)^k,
$$
by (\ref{t-special}) and (\ref{H-unitary-orthogonality}).
Since this product is not zero, it follows that the subrepresentation of 
${\cal H}^+ \otimes {\cal H}^+$ generated by $1 \otimes (z'_{11})^k$
contains $(\rho_{k+1},\Zh^+\otimes \BB C^{(k+1) \times (k+1)})$.
Therefore, each irreducible component of ${\cal H}^+ \otimes {\cal H}^+$
is contained in the subrepresentation of ${\cal H}^+ \otimes {\cal H}^+$
generated by $1 \otimes (z'_{11})^k$, $k=0,1,2,\dots$.
\end{proof}

\subsection{The Case of Ladder Diagrams}

In this subsection we prove Theorem \ref{main-thm} in the special case of
ladder diagrams. We label the variables as in Figure \ref{ladder-var-label}.
Since the function $l^{(n)}(Z_1,Z_2;W_1,W_2)$ is harmonic in $Z_2$,
the pairings (\ref{H-pairing}) and (\ref{H-pairing2}) agree,
and we can rewrite $L^{(n)}$ as
\begin{multline*}
L^{(n)} (\phi_1 \otimes \phi_2)(W_1,W_2) \\
= \frac{i}{4\pi^5} \iint_{\genfrac{}{}{0pt}{}{Z_1 \in U(2)_{R_1}}{Z_2 \in S^3_{R_2}}}
l^{(n)}(Z_1,Z_2;W_1,W_2) \cdot (\degt_{Z_1} \phi_1)(Z_1)
\cdot (\degt_{Z_2} \phi_2)(Z_2) \,\frac{dV_1}{N(Z_1)} \frac{dS_2}{R_2},
\end{multline*}
where $\phi_1, \phi_2 \in {\cal H}^+$, $R_1 > r_{\text{max},1}$,
$R_2 > r_{\text{max},2}$, $W_1 \in \BB D^+_{r_{\text{min},1}}$,
$W_2 \in \BB D^+_{r_{\text{min},2}}$, as before.

\begin{figure}
\begin{center}
\centerline{\includegraphics[scale=1.5]{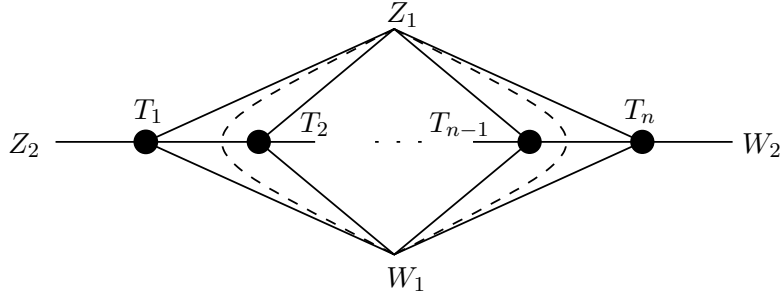}}
\vskip-36mm \hskip2mm $Z_1$

\vskip8mm $T_1$ \hskip60mm $T_n$

\vskip-3mm $T_2$ \hskip12mm $T_{n-1}$

\vskip-2mm $Z_2$ \hskip92mm $W_2$

\vskip13mm \hskip3mm $W_1$
\end{center}
\caption{$n$-loop ladder diagram (coordinate space variable).}
\label{ladder-var-label}
\end{figure}


\begin{lem}  \label{gen-val-lemma}
The operator $L^{(n)}$ associated to the $n$-loop ladder diagram sends
each $1 \otimes (z'_{11})^k$, $k=0,1,2,\dots$, into
\begin{equation}  \label{a-def}
\sum_{p=0}^k a^k(n,p) \cdot (w_{11})^{k-p} \cdot (w'_{11})^p,
\end{equation}
where the coefficients $a^k(n,p)$, $0 \le p \le k$, can be computed
from the recursive relations (\ref{a-ini}) and (\ref{a-rec-rel}).
In particular, $L^{(n)} (1 \otimes (z'_{11})^k)$ lies in
${\cal H}^+ \otimes {\cal H}^+$.
\end{lem}

\begin{proof}
One part of evaluating $L^{(n)} (1 \otimes (z'_{11})^k)$ is integrating over
$Z_1 \in U(2)_{R_1}$. First, we determine the effect of doing that.
Thus we integrate
\begin{equation}  \label{Z_1-integral}
\frac{i}{2\pi^3} \int_{Z_1 \in U(2)_{R_1}} \frac{N(Z_1-W_1)^{n-1}}
{N(Z_1-T_1) \cdot N(Z_1-T_2) \cdot \dots \cdot N(Z_1-T_n)} \, \frac{dV_1}{N(Z_1)},
\end{equation}
and we expand each $N(Z_1-T_j)^{-1}$ as in (\ref{1/N-expansion}):
$$
\frac 1{N(Z_1-T_j)}= N(Z_1)^{-1} \cdot \sum_{l,m,n}
t^l_{m\,\underline{n}}(T_j) \cdot t^l_{n\,\underline{m}}(Z_1^{-1}),
\qquad \begin{matrix} l=0,\frac 12, 1, \frac 32,\dots, \\
m,n = -l, -l+1, \dots, l. \end{matrix}
$$
Therefore,
$$
\frac1{N(Z_1-T_1) \cdot N(Z_1-T_2) \cdot \dots \cdot N(Z_1-T_n)} =
\frac1{N(Z_1)^n} + \text{ lower degree terms in $Z_1$}.
$$
On the other hand,
$$
N(Z_1-W_1)^{n-1} = N(Z_1)^{n-1} + \text{ lower degree terms in $Z_1$}.
$$
Comparing this with the orthogonality relations (\ref{orthogonality}),
we see that the integral (\ref{Z_1-integral}) is $1$.

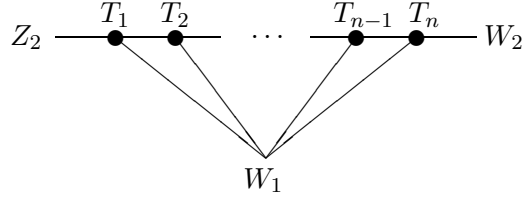
\begin{figure}
\centering\setlength{\unitlength}{1mm}
\begin{picture}(76,25)
\put(14,20){\circle*{2}}
\put(22,20){\circle*{2}}
\put(46,20){\circle*{2}}
\put(54,20){\circle*{2}}
\put(6,20){\line(1,0){22}}
\put(40,20){\line(1,0){22}}
\put(14,20){\line(5,-4){20}}
\put(54,20){\line(-5,-4){20}}
\put(22,20){\line(3,-4){12}}
\put(46,20){\line(-3,-4){12}}
\put(32,20){$\ldots$}
\put(0,19){$Z_2$}
\put(63,19){$W_2$}
\put(31,0){$W_1$}
\put(12,22){$T_1$}
\put(20,22){$T_2$}
\put(43,22){$T_{n-1}$}
\put(53,22){$T_n$}
\end{picture}
\caption{Reduced diagram.}
\label{reduced-diagram}
\end{figure}

Thus we end up integrating $(z'_{11})^k$ against something that
can be described by a reduced diagram in Figure \ref{reduced-diagram}.
When we integrate out the $Z_2$ variable, by Proposition \ref{quotient-prop},
we get
$$
\frac{i}{2\pi^3} \int_{Z_2 \in U(2)_{R_2}} \frac{\degt_{Z_2}((z'_{11})^k)}{N(Z_2-T_1)}
\,\frac{dV_2}{N(Z_2)} = (t_{11})^k.
$$
Then we integrate out the $T_1$ variable:
$$
\frac{i}{2\pi^3} \int_{T_1 \in U(2)_{r_1}} \frac{(t_{11})^k \,dV}
{N(W_1-T_1) \cdot N(T_1-T_2)}
$$
and, by Lemma \ref{z^p}, we get
$$
\sum_{p=0}^k a^n(1,p) \cdot (w_{11})^{k-p} \cdot (t'_{11})^p
$$
with each $a^k(1,p)$ given by (\ref{a-ini}).
Then we integrate the result against $\frac1{N(W_1-T_2) \cdot N(T_2-T_3)}$
and so on, until we integrate out the $T_n$ variable.
The recursive relation (\ref{a-rec-rel}) follows from Lemma \ref{z^p},
since the result of integration over $T_{n+1}$ has to be a linear
combination of $(w_{11})^{k-p} \cdot (w'_{11})^p$'s and the contribution to each
term $(w_{11})^{k-p} \cdot (w'_{11})^p$ comes precisely from the terms
$(w_{11})^{k-q} \cdot (t''_{11})^q$, $p \le q \le k$,
of the previous integration with weights $(q+1)^{-1}$.
\end{proof}

As a corollary of the proof, we also obtain:

\begin{cor}  \label{ladder-recursive-cor}
The operators $L^{(n)}$ associated to the $n$-loop ladder diagrams satisfy the
following recursive relation:
$$
L^{(n)} (1 \otimes (z'_{11})^k) = \frac1{k+1} \sum_{p=0}^k
(w_{11})^{k-p} \cdot L^{(n-1)} (1 \otimes (z'_{11})^p).
$$
\end{cor}

The proof of Theorem \ref{main-thm} would have been much easier if we knew
in advance that the operator $L^{(n)}$ is $\mathfrak{gl}(2,\HC)$-equivariant.
To deal with this issue, we introduce a closely related integral operator
for which $\mathfrak{gl}(2,\HC)$-equivariance is much easier to see.
Following notations from \cite{L}, we let
$$
\mathring{L}^{(n)} : (\varpi_2^l, \Zh) \otimes (\pi^0_r, {\cal H}^+) \to
(\pi^0_l, \Zh^+) \otimes (\pi^0_r, {\cal H}^+),
$$
$$
\mathring{L}^{(n)} (f \otimes \phi)(W_1,W_2)
= \frac{i}{4\pi^5} \iint_{\genfrac{}{}{0pt}{}{Z_1 \in U(2)_{R_1}}{Z_2 \in S^3_{R_2}}}
l^{(n)}(Z_1,Z_2;W_1,W_2) \cdot f(Z_1) \cdot (\degt_{Z_2} \phi)(Z_2)
\,dV_1\,\frac{dS_2}{R_2},
$$
where $f \in \Zh$, $\phi \in {\cal H}^+$, $R_1 > r_{\text{max},1}$,
$R_2 > r_{\text{max},2}$, $W_1 \in \BB D^+_{r_{\text{min},1}}$,
$W_2 \in \BB D^+_{r_{\text{min},2}}$, as before.
The $\mathfrak{gl}(2,\HC)$-equivariance of this operator follows from
the $\mathfrak{gl}(2,\HC)$-invariance of the bilinear pairings
(\ref{pairing}), (\ref{pairing2}) and Lemma \ref{conformal}.
Clearly, we have
\begin{equation}  \label{LL2}
L^{(n)} = \mathring{L}^{(n)} \circ \bigl( N(Z_1)^{-1} \cdot \degt_{Z_1} \bigr).
\end{equation}

\begin{lem}  \label{annihilator-lem1}
The operator $\mathring{L}^{(n)}$ annihilates $I_2^- \otimes {\cal H}^+$.
\end{lem}

\begin{proof}
Consider a pure tensor $f \otimes \phi$ with $f \in I_2^-$ and
$\phi \in {\cal H}^+$. Then $f(Z_1)$ is a sum of terms
$f_l(Z_1) \cdot N(Z_1)^{-l}$ with $f_l \in {\cal H}^+$ and $l \ge 2$.
Without loss of generality we can assume that each $f_l$ is homogeneous.
As in the proof of Lemma \ref{gen-val-lemma}, we observe that,
as a part of evaluating
$\mathring{L}^{(n)} \bigl( f_l(Z_1) \cdot N(Z_1)^{-l} \cdot \phi(Z_2) \bigr)$,
one needs to integrate over $Z_1 \in U(2)_{R_1}$:
\begin{equation}  \label{Z_1-integral-2}
\frac{i}{2\pi^3} \int_{Z_1 \in U(2)_{R_1}}
\frac{f_l(Z_1) \cdot N(Z_1)^{-l} \cdot N(Z_1-W_1)^{n-1}}
{N(Z_1-T_1) \cdot N(Z_1-T_2) \cdot \dots \cdot N(Z_1-T_n)} \,dV_1.
\end{equation}
Expanding each $N(Z_1-T_j)^{-1}$ as in (\ref{1/N-expansion}), we get
$$
\frac{N(Z_1)^{-l}}{N(Z_1-T_1) \cdot N(Z_1-T_2) \cdot \dots \cdot N(Z_1-T_n)} =
\frac1{N(Z_1)^{n+l}} + \text{ lower degree terms in $Z_1$}.
$$
On the other hand,
$$
f_l(Z_1) \cdot N(Z_1-W_1)^{n-1}
= f_l(Z_1) \cdot N(Z_1)^{n-1} + \text{ lower degree terms in $Z_1$}.
$$
Comparing this with orthogonality relations (\ref{orthogonality}),
since $l \ge 2$, we see that the integral (\ref{Z_1-integral-2}) and hence
$\mathring{L}^{(n)} \bigl( f_l(Z_1) \cdot N(Z_1)^{-l} \otimes \phi(Z_2) \bigr)$
are $0$.
\end{proof}

Let $\mathfrak{V} \subset \Zh \otimes {\cal H}^+$ denote the
subrepresentation of $(\varpi_2^l, \Zh) \otimes (\pi^0_r, {\cal H}^+)$
generated by
$$
\bigl\{ N(Z)^{-1} \cdot (z'_{11})^k ;\: k=0,1,2,3,\dots \bigr\}.
$$
Thus we have a $\mathfrak{gl}(2,\HC)$-equivariant map
$$
\mathring{L}^{(n)} : (\varpi_2^l \otimes \pi^0_r, \mathfrak{V})
\to (\pi^0_l, {\cal H}^+) \otimes (\pi^0_r, {\cal H}^+).
$$

\begin{lem}  \label{annihilator-lem2}
The operator $\mathring{L}^{(n)}$ annihilates
$\mathfrak{V} \cap (\Zh^+ \otimes {\cal H}^+)$.
\end{lem}

\begin{proof}
We proceed as in the proof of Lemma 28 in \cite{L}.
Observe that the operator $\mathring{L}^{(n)}$ increases the total
degree of an element of $\Zh \otimes {\cal H}^+$ by 2.
Now, suppose that there exists an element
$x \in \mathfrak{V} \cap (\Zh^+ \otimes {\cal H}^+)$ such that
$\mathring{L}^{(n)}(x) \ne 0$. Since $\mathring{L}^{(n)}$ is
$\mathfrak{gl}(2,\HC)$-equivariant, without loss of generality we can
assume that $\mathring{L}^{(n)}(x)$ belongs to one of the irreducible components
of $(\pi^0_l, {\cal H}^+) \otimes (\pi^0_r, {\cal H}^+)$.
Furthermore, we may assume that
$$
\mathring{L}^{(n)}(x) = (z_{ij}-z'_{ij})^k \qquad \text{for some }
x \in \mathfrak{V} \cap (\Zh^+ \otimes {\cal H}^+),\qquad k=0,1,2,\dots.
$$
Since $(z_{ij}-z'_{ij})^k$ is homogeneous of degree $k$, only the homogeneous
component $x'$ of degree $k-2$ of $x$ contributes anything to
$\mathring{L}^{(n)}(x)$, and $x' \in \Zh^+ \otimes {\cal H}^+$.

Now, let us regard $\mathring{L}^{(n)}$ as a $U(2) \times U(2)$ equivariant
map $(\varpi_2^l, \Zh^+) \otimes (\pi^0_l, {\cal H}^+) \to
(\pi^0_l, \Zh^+) \otimes (\pi^0_r, {\cal H}^+)$.
We have:
$$
\mathring{L}^{(n)}(x') = (z_{ij}-z'_{ij})^k \quad
\in V_{\frac{k}2} \boxtimes V_{\frac{k}2}.
$$
Since the degree of $x'$ is $k-2$,
$$
x' \in \bigoplus_{\genfrac{}{}{0pt}{}{2l+2p+2l'=k-2}{p,l,l' \ge 0}}
N(Z)^p \cdot (V_l \boxtimes V_l) \otimes (V_{l'} \boxtimes V_{l'}).
$$
But $V_l \otimes V_{l'}$ does not contain $V_{\frac{k}2}$ unless $l+l' \ge k/2$,
which produces a contradiction.
\end{proof}

Combining Lemmas \ref{annihilator-lem1} and \ref{annihilator-lem2}, we see that
$\mathring{L}^{(n)}$ descends to a well-defined
$\mathfrak{gl}(2,\HC)$-equivariant quotient map
\begin{equation}  \label{quotient-map}
\frac{\mathfrak{V}}{\mathfrak{V} \cap ((I^-_2 \oplus \Zh^+) \otimes {\cal H}^+)}
\to {\cal H}^+ \otimes {\cal H}^+.
\end{equation}
Clearly, this quotient space is a $\mathfrak{gl}(2,\HC)$-invariant subspace of 
$(\Zh/(I^-_2 \oplus \Zh^+)) \otimes {\cal H}^+$.
By Proposition \ref{quotient-prop} and Lemma \ref{product-lemma}, we have
the following isomorphisms of representations of $\mathfrak{gl}(2,\HC)$:
$$
\biggl( \varpi_2^l \otimes \pi^0_r, \frac{\mathfrak{V}}{\mathfrak{V} \cap
((I^-_2 \oplus \Zh^+) \otimes {\cal H}^+)} \biggr) \simeq
\biggl( \varpi_2^l \otimes \pi^0_r, \frac{\Zh}{I^-_2 \oplus \Zh^+}
\otimes {\cal H}^+\biggr)
\simeq (\pi^0_l, {\cal H}^+) \otimes (\pi^0_r, {\cal H}^+).
$$
From (\ref{LL2}) and Lemma \ref{gen-val-lemma} we conclude that
the operator $L^{(n)}$ has image in ${\cal H}^+ \otimes {\cal H}^+$
and the map (\ref{Ln}) is $\mathfrak{gl}(2,\HC)$-equivariant.

Since the irreducible components in the decomposition (\ref{tensor-decomp})
are pairwise distinct, by Schur's Lemma, $L^{(n)}$ must act on each irreducible
component $(\rho_k,\Zh^+\otimes \BB C^{k \times k})$ by multiplication by some
scalar $\mu^{(n)}_k$. Since $(w_{11}-w'_{11})^{k-1}$ generates
$\Zh^+\otimes \BB C^{k \times k}$, these scalars can be found by computing
the ratio of inner products
\begin{multline*}
\mu^{(n)}_k = \frac{\bigl\langle L^{(n)}(1 \otimes (z'_{11})^{k-1}),
(w_{11}-w'_{11})^{k-1} \bigr\rangle_{inn.\: prod.}}
{\bigl\langle 1 \otimes (w'_{11})^{k-1}, (w_{11}-w'_{11})^{k-1}
\bigr\rangle_{inn.\: prod.}}  \\
= \sum_{p=0}^{k-1} (-1)^{k+p+1} \cdot \begin{pmatrix} k-1 \\ p \end{pmatrix}
\cdot \bigl\langle L^{(n)}(1 \otimes (z'_{11})^{k-1}),
(w_{11})^{k-p-1} \cdot (w'_{11})^p \bigr\rangle_{inn.\: prod.}.
\end{multline*}
This sum can be evaluated using Lemma \ref{gen-val-lemma},
equation (\ref{t-special}) and orthogonality relationships
(\ref{H-unitary-orthogonality}).
This finishes the proof of Theorem \ref{main-thm} in the special case
of ladder diagrams.

\subsection{The General Case}

Now we prove Theorem \ref{main-thm} in complete generality, where
$L^{(n)}$ is the operator associated to any $n$-loop box diagram $d^{(n)}$.
The proof is by induction on the number of loops (or slingshots).
The case of the one-loop diagram (Figure \ref{12ladder}) was done in
\cite{FL1} and \cite{FL3}. So, suppose that the diagram $d^{(n)}$ is
obtained by adding a slingshot to an $(n-1)$-loop box diagram $d^{(n-1)}$.
As was mentioned before, there are exactly four ways of doing that
-- so that one of $Z_1$, $Z_2$, $W_1$ or $W_2$ becomes a solid vertex.
For concreteness, let us assume that the slingshot is attached to $d^{(n-1)}$
so that $Z_1$ becomes a solid vertex, the other cases are similar and easier.

Let $\tilde d^{(n)}$ be the $n$-loop ladder diagram
(Figure \ref{ladder-var-label}), and let $\tilde L^{(n)}$ be the
corresponding integral operator.
First, we prove the following symmetry property for $\tilde L^{(n)}$
(when $n=2$ it is a direct analogue of equation (8) in \cite{DHSS}).

\begin{lem}  \label{symmetry-lem}
The operator $\tilde L^{(n)}: (\pi^0_l, {\cal H}^+) \otimes (\pi^0_r, {\cal H}^+)
\to (\pi^0_l, {\cal H}^+) \otimes (\pi^0_r, {\cal H}^+)$
has the following symmetry:
$$
\tilde L^{(n)} (\phi_1 \otimes \phi_2)(W_1,W_2)
= \tilde L^{(n)} (\phi_2 \otimes \phi_1)(W_2,W_1),
\qquad \phi_1,\phi_2 \in {\cal H}^+.
$$
\end{lem}

\begin{proof}
Clearly, this property is true for the generators $(z_{11}-z'_{11})^k$,
$k \ge 0$, of ${\cal H}^+ \otimes {\cal H}^+$.
Therefore, by the $\mathfrak{gl}(2,\HC)$-equivariance of $\tilde L^{(n)}$,
it is true for all elements of ${\cal H}^+ \otimes {\cal H}^+$.
\end{proof}

By induction, we assume that Theorem \ref{main-thm} --
and hence Corollary \ref{magic-id-operators} -- hold for $n-1$.

\begin{lem}
We have:
\begin{equation}  \label{gen-equal-eqn}
L^{(n)} ((z_{11})^k \otimes 1) = \tilde L^{(n)} ((z_{11})^k \otimes 1),
\qquad k =0,1,2,\dots.
\end{equation}
\end{lem}

\begin{proof}
Similarly to the proof of Lemma \ref{gen-val-lemma}, we first integrate over
$Z_2 \in U(2)_{R_2}$.
It is easy to see that the number of solid edges at $Z_2$ is one more than
the number of dashed edges.
Thus the integral over $Z_2$ is very similar to the integral
(\ref{Z_1-integral}) with variable $Z_1$ replaced with $Z_2$;
this integral is 1.
Hence, for the purposes of evaluating $L^{(n)} ((z_{11})^k \otimes 1)$,
in diagram $d^{(n)}$ we can delete the solid edge joining $Z_2$ and $T_n$
and the dashed edge joining $Z_2$ and $W_2$ (see Figure \ref{simplified-fig}).

\begin{figure}
\begin{center}
\begin{subfigure}{0.16\textwidth}
\centering
\includegraphics[scale=1]{slingshot.eps}
\end{subfigure}
\hspace{-0.18\textwidth}
\begin{subfigure}{0.16\textwidth}
\centering
\setlength{\unitlength}{1mm}
\begin{picture}(22,28)
\put(0,13){$Z_1$}
\put(17,0){$Z_2$}
\put(17,25){$W_2$}
\put(9,16){$T_n$}
\end{picture}
\end{subfigure}
+ \quad $d^{(n-1)}$
\qquad
$\begin{matrix} \text{can be} \\ \text{replaced with} \end{matrix}$
\qquad
\begin{subfigure}{0.16\textwidth}
\centering
\vskip-.36in
\includegraphics[scale=1]{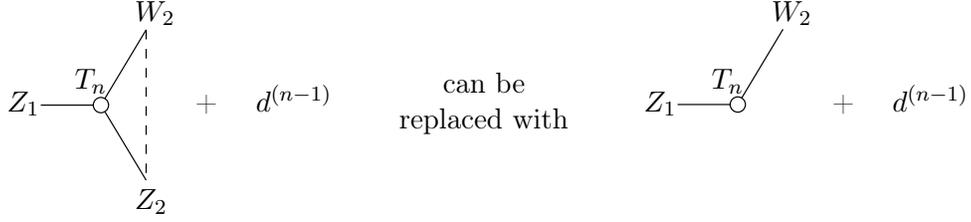}
\end{subfigure}
\hspace{-0.18\textwidth}
\begin{subfigure}{0.16\textwidth}
\centering
\setlength{\unitlength}{1mm}
\begin{picture}(22,28)
\put(0,13){$Z_1$}
\put(17,25){$W_2$}
\put(9,16){$T_n$}
\end{picture}
\end{subfigure}
+ \quad $d^{(n-1)}$
\end{center}
\caption{We can make this replacement in $d^{(n)}$
for the purposes of evaluating $L^{(n)} ((z_{11})^k \otimes 1)$.}
\label{simplified-fig}
\end{figure}

When we integrate out the $Z_1$ variable, by Proposition \ref{quotient-prop},
we get
$$
\frac{i}{2\pi^3} \int_{Z_1 \in U(2)_{R_1}} \frac{\degt_{Z_1}((z_{11})^k)}{N(Z_1-T_n)}
\,\frac{dV_1}{N(Z_1)} = (t''_{11})^k = t^{k/2}_{-k/2\,\underline{-k/2}}(T_n).
$$
Then we integrate out the $T_n$ variable:
\begin{equation}  \label{T_n-integral}
\frac{i}{2\pi^3} \int_{T_n \in U(2)_{r_1}} t^{k/2}_{-k/2\,\underline{-k/2}}(T_n) \cdot
\frac{l^{(n-1)}(T_n,Z_2;W_1,W_2)}{N(W_2-T_n)} \,dV.
\end{equation}
Using (\ref{1/N-expansion}), we expand $N(W_2-T_n)^{-1}$ and
$l^{(n-1)}(T_n,Z_2;W_1,W_2)$ in terms of basis functions (\ref{Zh-basis}):
$$
\frac 1{N(W_2-T_n)}= N(T_n)^{-1} \cdot \sum_{l,p,q}
t^l_{p\,\underline{q}}(T_n^{-1}) \cdot t^l_{q\,\underline{p}}(W_2),
\qquad \begin{matrix} l=0,\frac 12, 1, \frac 32,\dots, \\
p,q = -l, -l+1, \dots, l, \end{matrix}
$$
$$
l^{(n-1)}(T_n,Z_2;W_1,W_2) = \sum_{l',p',q',m}
t^{l'}_{p'\,\underline{q'}}(T_n^{-1}) \cdot N(T_n)^{-1-m} \cdot f_{l',p',q',m}(Z_2;W_1,W_2)
$$
for some functions $f_{l',p',q',m}(Z_2;W_1,W_2)$.
In the diagram $d^{(n-1)}$, the number of solid edges at $Z_1$ is one more than
the number of dashed edges.
This implies that only the terms with $m \ge 0$ appear in the expansion of
$l^{(n-1)}(T_n,Z_2;W_1,W_2)$.
Thus the integral (\ref{T_n-integral}) can be rewritten as
$$
\sum_{\genfrac{}{}{0pt}{}{l,p,q}{l',p',q',m}}
\bigl\langle N(T_n)^{-2-m} \cdot t^l_{p \, \underline{q}}(T_n^{-1}) \cdot
t^{l'}_{p' \, \underline{q'}}(T_n^{-1}), t^{k/2}_{-k/2\,\underline{-k/2}}(T_n) \bigr\rangle
\cdot t^l_{q\,\underline{p}}(W_2) \cdot f_{l',p',q',m}(Z_2;W_1,W_2).
$$
By Corollary \ref{z^p-cor}, these terms are zero unless $m=0$, $l+l'=k/2$,
$p=q=-l$ and $p'=q'=-l'$, and (\ref{T_n-integral}) becomes
\begin{multline*}
\frac1{k+1} \sum_{l+l'=k/2} t^l_{-l\,\underline{-l}}(W_2) \cdot
f_{l',-l',-l',0}(Z_2;W_1,W_2) \\
= \sum_{l,l',p',q',m} \frac{2l+1}{k+1} t^l_{-l\,\underline{-l}}(W_2)
\cdot \bigl\langle t^{l'}_{-l'\,\underline{-l'}}(T_n),
t^{l'}_{p' \, \underline{q'}}(T_n^{-1}) \cdot N(T_n)^{-2-m} \bigr\rangle
\cdot f_{l',p',q',m}(Z_2;W_1,W_2) \\
= \frac1{k+1} \sum_{l+l'=k/2} t^l_{-l\,\underline{-l}}(W_2) \cdot
\frac{i}{2\pi^3} \int_{T_n \in U(2)_{r_n}} l^{(n-1)}(T_n,Z_2;W_1,W_2) \cdot
\degt_{T_n}\bigl( t^{l'}_{-l'\,\underline{-l'}}(T_n) \bigr) \,\frac{dV_{T_n}}{N(T_n)}.
\end{multline*}
This proves that
$$
L^{(n)} ((z_{11})^k \otimes 1) = \frac1{k+1} \sum_{p=0}^k
(w'_{11})^{k-p} \cdot L^{(n-1)} ((z_{11})^p \otimes 1),
$$
where $L^{(n-1)}$ denotes the integral operator corresponding to $d^{(n-1)}$.
By induction hypothesis, Corollary \ref{ladder-recursive-cor} and
Lemma \ref{symmetry-lem}, this implies (\ref{gen-equal-eqn}).
\end{proof}

From this point on, the proof proceeds as in the ladder case,
with trivial modifications.
Because of the way the last slingshot is attached, we know that the
function $l^{(n)}(Z_1,Z_2;W_1,W_2)$ is harmonic in $Z_1$.
Then the pairings (\ref{H-pairing}) and (\ref{H-pairing2}) agree,
and we can rewrite $L^{(n)}$ as
\begin{multline*}
L^{(n)} (\phi_1 \otimes \phi_2)(W_1,W_2) \\
= \frac{i}{4\pi^5} \iint_{\genfrac{}{}{0pt}{}{Z_1 \in S^3_{R_1}}{Z_2 \in U(2)_{R_2}}}
l^{(n)}(Z_1,Z_2;W_1,W_2) \cdot (\degt_{Z_1} \phi_1)(Z_1)
\cdot (\degt_{Z_2} \phi_2)(Z_2) \,\frac{dS_1}{R_1} \frac{dV_2}{N(Z_2)},
\end{multline*}
where $\phi_1, \phi_2 \in {\cal H}^+$, $R_1 > r_{\text{max},1}$,
$R_2 > r_{\text{max},2}$, $W_1 \in \BB D^+_{r_{\text{min},1}}$,
$W_2 \in \BB D^+_{r_{\text{min},2}}$, as before.
We introduce a closely related integral operator
$$
\mathring{L}^{(n)} : (\pi^0_l, {\cal H}^+) \otimes (\varpi_2^r, \Zh) \to
(\pi^0_l, \Zh^+) \otimes (\pi^0_r, \Zh^+),
$$
$$
\mathring{L}^{(n)} (\phi \otimes f)(W_1,W_2)
= \frac{i}{4\pi^5} \iint_{\genfrac{}{}{0pt}{}{Z_1 \in S^3_{R_1}}{Z_2 \in U(2)_{R_2}}}
l^{(n)}(Z_1,Z_2;W_1,W_2) \cdot (\degt_{Z_1} \phi)(Z_1) \cdot f(Z_2)
\,\frac{dS_1}{R_1}\,dV_2,
$$
where $\phi \in {\cal H}^+$, $f \in \Zh$.
The $\mathfrak{gl}(2,\HC)$-equivariance of this operator follows from
the $\mathfrak{gl}(2,\HC)$-invariance of the bilinear pairings
(\ref{pairing}), (\ref{pairing2}) and Lemma \ref{conformal}.
Clearly, we have
\begin{equation}  \label{LL2-gen}
L^{(n)} = \mathring{L}^{(n)} \circ \bigl( N(Z_2)^{-1} \cdot \degt_{Z_2} \bigr).
\end{equation}

\begin{lem}  \label{annihilator-lem1-gen}
The operator $\mathring{L}^{(n)}$ annihilates ${\cal H}^+ \otimes I_2^-$.
\end{lem}

\begin{proof}
As we have observed earlier, the number of solid edges at $Z_2$ is one more
than the number of dashed edges.
Using this observation, one proceeds exactly as in the proof of
Lemma \ref{annihilator-lem1} with the roles of the variables $Z_1$ and $Z_2$
switched.
\end{proof}

Let $\mathfrak{V}' \subset {\cal H}^+ \otimes \Zh$ denote the
subrepresentation of $(\pi^0_l, {\cal H}^+) \otimes (\varpi_2^r, \Zh)$
generated by
$$
\bigl\{ (z_{11})^k \cdot N(Z')^{-1} ;\: k=0,1,2,3,\dots \bigr\}.
$$
Thus we have a $\mathfrak{gl}(2,\HC)$-equivariant map
$$
\mathring{L}^{(n)} : (\pi^0_l \otimes \varpi_2^r, \mathfrak{V}')
\to (\pi^0_l, {\cal H}^+) \otimes (\pi^0_r, {\cal H}^+).
$$
The same proof as that of Lemma \ref{annihilator-lem2} also shows:

\begin{lem}  \label{annihilator-lem2-gen}
The operator $\mathring{L}^{(n)}$ annihilates
$\mathfrak{V}' \cap ({\cal H}^+ \otimes \Zh^+)$.
\end{lem}

Combining Lemmas \ref{annihilator-lem1-gen} and \ref{annihilator-lem2-gen},
we see that $\mathring{L}^{(n)}$ descends to a well-defined
$\mathfrak{gl}(2,\HC)$-equivariant quotient map
\begin{equation}  \label{quotient-map-gen}
\frac{\mathfrak{V}'}{\mathfrak{V}' \cap
({\cal H}^+ \otimes (I^-_2 \oplus \Zh^+))} \to {\cal H}^+ \otimes {\cal H}^+.
\end{equation}
Clearly, this quotient space is a $\mathfrak{gl}(2,\HC)$-invariant subspace of 
${\cal H}^+ \otimes (\Zh/(I^-_2 \oplus \Zh^+))$.
By Proposition \ref{quotient-prop} and Lemma \ref{product-lemma}, we have
the following isomorphisms of representations of $\mathfrak{gl}(2,\HC)$:
$$
\biggl( \pi^0_l \otimes \varpi_2^r, \frac{\mathfrak{V}'}{\mathfrak{V}' \cap
({\cal H}^+ \otimes (I^-_2 \oplus \Zh^+))} \biggr) \simeq
\biggl( \pi^0_l \otimes \varpi_2^r, {\cal H}^+ \otimes
\frac{\Zh}{I^-_2 \oplus \Zh^+} \biggr)
\simeq (\pi^0_l, {\cal H}^+) \otimes (\pi^0_r, {\cal H}^+).
$$
From (\ref{LL2-gen}) and Lemma \ref{gen-val-lemma} we conclude that
the operator $L^{(n)}$ has image in ${\cal H}^+ \otimes {\cal H}^+$
and the map (\ref{Ln}) is $\mathfrak{gl}(2,\HC)$-equivariant.
By (\ref{gen-equal-eqn}), the maps $L^{(n)}$ and $\tilde L^{(n)}$ coincide
on the generators of ${\cal H}^+ \otimes {\cal H}^+$.
Since Theorem \ref{main-thm} is already established for $\tilde L^{(n)}$,
it follows that the result holds for $L^{(n)}$ as well.

\separate

\separate

\noindent
{\em Department of Mathematics, Indiana University,
Rawles Hall, 831 East 3rd St, Bloomington, IN 47405}

\end{document}